\documentclass[final,onefignum,onetabnum]{siamart190516}

\usepackage{amsfonts}
\usepackage{supertabular}
\usepackage{adjustbox}
\usepackage{array}
\newcolumntype{R}[2]{%
    >{\adjustbox{angle=#1,lap=\width-(#2)}\bgroup}%
    l%
    <{\egroup}%
}
\newcommand *\rot{\multicolumn{1}{R{45}{1em}}}

\usepackage{pifont}
\newcommand{\cmark}{\ding{51}}%
\newcommand{\xmark}{\ding{55}}%

\usepackage{float}
\usepackage{enumitem}
\usepackage{hyperref}

\newcommand{\rank}{\hbox{rank}}
\newcommand{\norm}[1]{\left\Vert#1\right\Vert}
\newcommand{\abs}[1]{\left\vert#1\right\vert}

\newsiamthm{example}{Example} 
\newsiamremark{remark}{Remark} 
\newsiamremark{observation}{Observation} 
\newsiamthm{jdef}{Definition} 
\crefname{remark}{Remark}{Remarks}

\makeatletter
\newcommand{\labitem}[2]{%
\def\@itemlabel{\textbf{#1}}
\item
\def\@currentlabel{#1}\label{#2}}
\makeatother

\headers{Local search and generalized inverses}{L. Xu, M. Fampa, J. Lee and G. Ponte}

\title{Approximate 1-norm minimization and minimum-rank structured sparsity
for various  generalized inverses via local search\thanks{Submitted to the editors \today.
M. Fampa was supported in part by CNPq grant 303898/2016-0.
J. Lee was supported in part by ONR grant N00014-17-1-2296.}}

\author{Luze Xu\thanks{University of Michigan, Ann Arbor, MI, USA (\email{xuluze,jonxlee@umich.edu})}
\and Marcia Fampa\thanks{Universidade Federal do Rio de Janeiro
  (\email{fampa@cos.ufrj.br, gabrielponte@poli.ufrj.br})},
  \and Jon Lee$\strut^\dag$
  \and Gabriel Ponte$\strut^\ddag$}

\begin{document}

\maketitle

\begin{abstract}
Fundamental in matrix algebra and its applications, a \emph{generalized inverse} of a real matrix $A$
is a matrix $H$  that satisfies the Moore-Penrose (M-P) property
$AHA=A$. If $H$ also satisfies the additional useful M-P property, $HAH=H$, it is
called a \emph{reflexive generalized inverse}. Reflexivity is equivalent to minimum rank, so
we are particularly interested in reflexive generalized inverses.
We consider aspects of symmetry related to
the calculation of a \emph{sparse} reflexive generalized inverse of $A$.
As is common, and following Lee and Fampa (2018) for calculating sparse generalized inverses,
we use (vector) 1-norm minimization for inducing sparsity and for keeping the magnitude
of entries under control.

When $A$ is symmetric, we may naturally desire a symmetric $H$; while generally such a restriction
on $H$ may not lead to a 1-norm minimizing reflexive generalized inverse.
We investigate a block construction method to produce a symmetric reflexive generalized inverse that is structured and has guaranteed sparsity.
We provide a theoretically-efficient and practical local-search algorithm to block-construct an approximate 1-norm minimizing symmetric reflexive generalized inverse.

Another aspect of symmetry that we consider relates to another M-P property:
$H$ is \emph{ah-symmetric} if $AH$ is symmetric. The ah-symmetry property is
the key one for solving least-squares problems using $H$.
Here we do not assume that $A$ is symmetric, and we do not impose symmetry on $H$.
We investigate a column
block construction method to produce an ah-symmetric  reflexive generalized inverse
that is structured and has guaranteed sparsity.
We provide a theoretically-efficient and practical local-search algorithm to column block construct an approximate 1-norm minimizing ah-symmetric reflexive generalized inverse.
\end{abstract}

\begin{keywords}
generalized inverse;
sparse optimization;
approximation algorithm
\end{keywords}

\begin{AMS}
Primary: 90C26,  
90C25;  
secondary: 15A09, 
65K05  
\end{AMS}


\section{Introduction}\label{sec:intro}

Generalized inverses are essential tools in matrix algebra and its applications. In particular, the Moore-Penrose (M-P) pseudoinverse can be used to calculate the least-squares solution of an over-determined system of linear equations and the solution with minimum $2$-norm of an under-determined system of linear equations.
In both cases, if the system  is $Ax=b$, then a solution is
given by $x:=A^+b$, where $A^+$ is the M-P pseudoinverse.
Considering our motivating \emph{use case} of a very large (rank deficient) matrix $A$
and multiple right-hand sides $b$, we can  see the value of having at hand a sparse
generalized inverse.
We apply techniques of sparse optimization, aiming at balancing the tradeoff between properties of the M-P pseudoinverse and alternative sparser generalized inverses.
Recently, \cite{dokmanic,dokmanic1,dokmanic2} used sparse-optimization techniques to give tractable right and left sparse pseudoinverses. Particularly relevant to what we present here, \cite{FFL2016} (also see \cite{FFL2019}) derived and analyzed other tractable sparse generalized inverses based on relaxing some of the ``M-P properties''. \cite{FampaLee2018ORL} investigated one such kind of sparse generalized inverse, with particular interest in rank-deficient matrices;
these reduce to the sparse right (resp., left)  pseudoinverses in \cite{dokmanic,dokmanic1,dokmanic2}, when the matrix has full row (resp., column) rank.

In what follows, for succinctness, we use vector-norm notation on matrices: we write $\|H\|_1$ to mean $\|\mathrm{vec}(H)\|_1$, and $\|H\|_{\max}$ to mean $\|\mathrm{vec}(H)\|_{\max}$ (in both cases, these are not the usual induced/operator matrix norms). We use $I$ for an identity matrix and $J$ for an all-ones matrix. Matrix dot product is indicated by $\langle X, Y\rangle=\mathrm{trace}(X^\top Y):=\sum_{ij}x_{ij}y_{ij}$. We use $A[S,T]$ for the submatrix of $A$ with row indices $S$ and column indices $T$; additionally, we use $A[S,:]$ ( resp., $A[:,T]$) for the submatrix of $A$ formed by the rows $S$ (resp., columns $T$).
Finally, if $A$ is symmetric and $S=T$, we use $A[S]$ to represent the principal submatrix of $A$ with  row/column indices $S$.

When a real matrix $A\in\mathbb{R}^{m\times n}$ is not square or is square but not invertible, we consider ``pseudoinverses'' of $A$ (see \cite{rao1971}). The most well-known pseudoinverse is the \emph{M-P pseudoinverse}
(see \cite{Bjerhammar,Moore,Penrose}). If $A=U\Sigma V^\top$ is the real singular-value decomposition of $A$ (see \cite{GVL1996}, for example), where $U\in\mathbb{R}^{m\times m}$, $V\in\mathbb{R}^{n\times n}$ are orthogonal matrices and $\Sigma=\mathrm{diag}(\sigma_1,\sigma_2,\dots,\sigma_p)\in\mathbb{R}^{m\times n}$ ($p=\min\{m,n\}$) with  singular values $\sigma_1\ge\sigma_2\ge\dots\ge\sigma_p\ge0$, then the M-P pseudoinverse of $A$ can be defined as $A^+:=V\Sigma^+U^\top$, where $\Sigma^+:=\mathrm{diag}(\sigma_1^+,\sigma_2^+,\dots,\sigma_p^+)\in\mathbb{R}^{n\times m}$, $\sigma_i^+:=1/\sigma_i$ for all $\sigma_i\ne 0$, and $\sigma_i^+:=0$ for all $\sigma_i=0$. The M-P pseudoinverse plays a very important role in matrix theory and is widely-used in practice.

Following \cite{FampaLee2018ORL}, we define different tractable sparse ``generalized inverses'', based on the following very-well-known fundamental characterization of the M-P pseudoinverse.

\begin{theorem}[see \cite{Penrose}]
For $A \in \mathbb{R}^{m \times n}$, the M-P pseudoinverse $A^+$ is the unique $H \in \mathbb{R}^{n \times m}$ satisfying:
	\begin{align}
		& AHA = A \label{property1} \tag{P1}\\
		& HAH = H \label{property2} \tag{P2}\\
		& (AH)^{\top} = AH \label{property3} \tag{P3}\\
		& (HA)^{\top} = HA \label{property4} \tag{P4}
	\end{align}
\end{theorem}
Following \cite{RohdeThesis}, a \emph{generalized inverse} is any $H$ satisfying  \ref{property1}. Because we are interested in sparse $H$, \ref{property1} is important to enforce, otherwise the completely sparse zero-matrix
(which carries no information from $A$) always satisfies
\ref{property2}$+$\ref{property3}$+$\ref{property4}.
A generalized inverse is \emph{reflexive} if it satisfies \ref{property2} (see \cite{RohdeThesis}). Theorem 3.14 in \cite{RohdeThesis} tells us two very useful facts: (i) if $H$ is a generalized inverse of $A$, then $\mathrm{rank}(H)\ge\mathrm{rank}(A)$, and (ii) a generalized inverse $H$ of $A$ is reflexive if and only if $\mathrm{rank}(H)=\mathrm{rank}(A)$. A low-rank $H$ can be viewed as being more interpretable/explainable model (say in the context of
the least-squares problem), so we naturally prefer reflexive generalized inverses (which have the least rank possible among generalized inverses).
As we have said, we are interested in
sparse generalized inverses. But structured sparsity of $H$ is even more valuable, as
it can be viewed, in a different way, as being a more interpretable/explainable model.
Later, we will expand on this point, but essentially we prefer nonzeros that are confined to
a  block of $H$ having limited size.

As a convenient mnemonic, if $H$ satisfies \ref{property3}, we say that $H$ is \emph{ah-symmetric}, and if $H$ satisfies \ref{property4}, we say that $H$  is \emph{ha-symmetric}. That is, ah-symmetric (resp., ha-symmetric)
means that $AH$ (resp., $HA$) is symmetric.

It is very important to know that not  all of the M-P properties are required for a generalized inverse to exactly solve key problems. For example, if $H$ is an ah-symmetric generalized inverse, then $\hat{x}:=Hb$ solves $\min\{\|Ax-b\|_2:~x\in\mathbb{R}^n\}$; if $H$ is a ha-symmetric generalized inverse, then $\hat{x}:=Hb$ solves $\min\{\|x\|_2:~Ax=b,~x\in\mathbb{R}^n\}$ (see \cite{FFL2016,campbell2009generalized}). This is an extremely important point for us,
which we come to in \S\ref{sec:AHsym}.

It is hard to find a generalized inverse (i.e., a solution of \ref{property1}) having the minimum number of nonzeros, subject to various subsets of $\{$\ref{property2}, \ref{property3}, \ref{property4}$\}$ (but not all of them). We let $\|H\|_0$ (resp., $\|x\|_0$)
be the number of nonzeros in the matrix $H$ (resp., vector $x$). \cite{dokmanic2} established that $\min\{\|H\|_0:~\ref{property1}\}$ is {\bf NP}-hard as follows: for full row-rank $A\in\mathbb{R}^{m\times n}$ ($m<n$), we have
$
\min\{\|H\|_0:~AHA=A\}=\min\{\|H\|_0:~AH=I\},
$
and computing a minimizer can be done column-wise, as a collection of sparse optimization problems $\min\{\|x\|_0:~Ax=e_i\}$. These latter sparse optimization problems are
known to be {\bf NP}-hard (see \cite{natarajan1995sparse}) for a general right-hand side $b\not=0$. But with $A$ having full row rank, we can reduce any  general right-hand side $b\not=0$, to a problem with $b=e_i$, by left-multiplying $A$ and $b$ by an appropriate
square and invertible matrix. Using the same idea, we can show the following
hardness result.
\begin{proposition}\label{prop:hard}
The following problems are {\bf NP}-hard:
\begin{align}
&\min\{\|H\|_0:~\ref{property1}+\ref{property2}\}; \label{sginv12} \tag{$SGI12$}\\
&\min\{\|H\|_0:~\ref{property1}+\ref{property3}\}; \label{sginv13} \tag{$SGI13$}\\
&\min\{\|H\|_0:~\ref{property1}+\ref{property2}+\ref{property3}\}; \label{sginv123} \tag{$SGI123$}\\
&\min\{\|H\|_0:~\ref{property1}+\ref{property4}\} \label{sginv14}; \tag{$SGI14$}\\
&\min\{\|H\|_0:~\ref{property1}+\ref{property2}+\ref{property4}\}. \label{sginv124} \tag{$SGI124$}
\end{align}
\end{proposition}
\begin{proof}
For full row-rank $A\in\mathbb{R}^{m\times n}$ ($m<n$), we have $AHA=A~\Leftrightarrow AH=I$, and thus $HAH=H$ and $(AH)^\top=AH$ are also satisfied. Therefore, \eqref{sginv12}, \eqref{sginv13}, \eqref{sginv123} are all equivalent to $\min\{\|H\|_0:~AH=I\}$, which is {\bf NP}-hard. Similarly, with full column-rank  $A$, we have that \eqref{sginv14}, \eqref{sginv124} are {\bf NP}-hard.
\end{proof}
\smallskip

We note that we have not been able to resolve the complexity of
\begin{align}
&\min\{\|H\|_0:~\ref{property1}+\ref{property3}+\ref{property4}\}. \label{sginv134} \tag{$SGI134$}
\end{align}

 Because of \cref{prop:hard}, we take the standard approach of minimizing $\|H\|_1$ to induce sparsity, subject to \ref{property1}  and various subsets of $\{$\ref{property2}, \ref{property3}, \ref{property4}$\}$ (but not all).

 It is a very important point that minimizing  $\|H\|_1$  (or any norm),
 serves to keep the entries of $H$ under control. This is very useful
 for applications, because it leads to more reasonable models
 (e.g., in the least-squares application) and with better numerics.
Minimizing $\|H\|_0$ does not have any such property (as $\|\cdot\|_0$
is not a norm). Indeed, the cost of $10^{-8}$ and $10^8$ are  the same
under $\|\cdot\|_0$; but we can effectively round entries on the order of $10^{-8}$
to 0 in $H$, while many entries  on the order of  $10^8$ in $H$ will lead to
unstable computations using $H$.
 It might seem that minimizing $\|H\|_{\max}$
would more naturally keep \emph{entries} of $H$ under control,
but there is a strong preference for  minimizing  $\|\cdot\|_1$
because it empirically induces sparsity, and it captures the lower envelope of $\|\cdot\|_0$ when the argument entries are in $[-1,1]$. Moreover, $\|H\|_{\max}$
sees no benefit for reducing entries of $H$ that are not largest.

Considering the tractability of minimizing  $\|H\|_1$,
we see that \ref{property1}, \ref{property3} and \ref{property4} are linear constraints, which are easy to handle, while \ref{property2} is a non-convex quadratic, hence rather nasty. But, as we have noted, \ref{property2} is
very useful for a generalized inverse, as it is equivalent to the rank of $H$ being equal to the rank of $A$.
Therefore, we are particularly interested in situations where,
without solving a mathematical-programming formulation via a generic method
(like LP or non-convex quadratically-constrained programming),
we can construct a minimizer or approximate minimizer  of $\|H\|_1$,
subject to  \ref{property1}, \ref{property2},
 and one or none of \ref{property3} and \ref{property4}.
 In fact, our methods will do this and more.
 Additionally, we will get \emph{structured sparsity} for $H$.


\cite{FampaLee2018ORL} gave some results in this direction,
when neither \ref{property3} nor  \ref{property4} is enforced.
In particular, \cite{FampaLee2018ORL} gave a ``block construction'' of a generalized inverse $H$ of rank-$r$ $A$ that is always reflexive, is ``somewhat-sparse'', having at most $r^2$ nonzeros and all confined to a choice of $r$ rows and $r$ columns (hence, structured). We note that any generalized inverse of $A$ must have at least $r$ nonzeros (because its rank is always at least $r$). Therefore, for any choice of
block, the construction of \cite{FampaLee2018ORL} has the number of nonzeros
within a factor of $r$ of the minimum number of nonzeros.

\cite{FampaLee2018ORL} also demonstrated that  there exists an easy-to-find block construction of a  $1$-norm
minimizing reflexive generalized inverse, for rank-$1$ matrices and rank-$2$ nonnegative matrices. Finally, for general rank-$r$ matrices, \cite{FampaLee2018ORL} gave an efficient local-search based approximation algorithm, that efficiently finds a generalized inverse following the block construction, and that has its $1$-norm within a factor of (almost) $r^2$ of the minimum $1$-norm of any generalized inverse. In fact, experimentally, we see much better performance for the local search than this guarantee (see \cite{FLPX}), while we establish here that
the guarantee of the local search is best possible; see \S\ref{sec:examples}.

In what follows, we follow two directions. One direction aims at finding a sparse \emph{symmetric} reflexive generalized inverse $H$ for a symmetric matrix $A$.
Because the M-P pseudoinverse of a symmetric matrix is also symmetric, it is natural to ask for a symmetric reflexive generalized inverse.
\cite[Section 3.3]{RohdeThesis} demonstrates that if $A$ is symmetric, then it is not necessarily the case that a reflexive generalized inverse is symmetric;
but there always does exist a symmetric reflexive generalized inverse (e.g., the M-P pseudoinverse). \Cref{prop:symnphard} below establishes that for a symmetric matrix $A$, finding a symmetric generalized inverse with minimum number of nonzeros is {\bf NP}-hard. So we aim at
 construction of a symmetric reflexive generalized inverse with minimum (or approximately minimum) 1-norm.
\begin{proposition}\label{prop:symnphard}
For symmetric matrix $A$, the following problem is {\bf NP}-hard.
\begin{equation}
\min\{\|H\|_0:~\ref{property1},~H^\top=H\} \label{symsginv} \tag{$symSGI$}
\end{equation}
\end{proposition}
\begin{proof}
We reduce $\min\{\|H\|_0:~\ref{property1}\}$ to an instance of \eqref{symsginv} as follows. Let
$$
\bar{A}:=\begin{bmatrix}
0 & A\\
A^\top & 0
\end{bmatrix}
\mbox{ and }
H := \begin{bmatrix}
X & Z^\top\\
Z & Y
\end{bmatrix};
\mbox{ then }
\bar{A}H\bar{A} = \begin{bmatrix}
AYA^\top & AZA\\
A^\top Z^\top A^\top & A^\top X A{}
\end{bmatrix}.
$$
Thus $\bar{A}$ is symmetric, and \eqref{symsginv} for $\bar{A}$ is equivalent to
$$
\min\{\|X\|_0+\|Y\|_0+2\|Z\|_0:~A^\top X A=0, AYA^\top = 0, AZA=A, X^\top=X, Y^\top=Y\}.
$$
Clearly, the optimal solutions of $\eqref{symsginv}$ for $\bar{A}$ all have $X=0$, $Y=0$, and $(X=0,~ Y=0,~ Z)$ is optimal to $\eqref{symsginv}$ for $\bar{A}$ if and only if $Z$ is optimal to $\min\{\|H\|_0:~AHA=A\}$; thus \eqref{symsginv} is {\bf NP}-hard.
\end{proof}
\smallskip
Unfortunately, we do not know the complexity of $\min\{\|H\|_0:~\ref{property1}+\ref{property2},~H^\top=H\}$.

Our second direction aims at finding sparse ah-symmetric (or ha-symmetric) reflexive generalized inverses.
Note that if $A$ is symmetric, and we require that $H$ is a symmetric ah-symmetric (or ha-symmetric) reflexive generalized inverse,
then $H$ is already the M-P pseudoinverse (see \cite{RohdeThesis}).
 Therefore, there
is no interest in enforcing symmetry on $H$ in this context. \Cref{prop:hard} (\ref{sginv123}, \ref{sginv124}) establishes that finding an ah-symmetric (or ha-symmetric) reflexive generalized inverse with minimum number of nonzeros is {\bf NP}-hard even for the full row (or column) rank matrix $A$. So we aim at construction of an ah-symmetric (or ha-symmetric) reflexive generalized inverse with minimum (or approximately minimum) 1-norm. Unlike the symmetric case, a 1-norm minimizing ah-symmetric (or ha-symmetric) reflexive generalized inverse can be obtained by  recasting the problem as a linear-optimization problem. However, the block construction method can  be generalized to give an ah-symmetric (or ha-symmetric) reflexive generalized inverse with a better guaranteed sparsity in terms of the number of nonzeros.


In \S\ref{sec:sym}, we consider the situation where $A$ is symmetric.
We
give a local-search based (almost) $r^2$-approximation algorithm for finding a
$1$-norm minimizing symmetric reflexive generalized inverse.
Along the way, we repair a proof of a key result from \cite{FampaLee2018ORL}, concerning the correctness of the approximation algorithm.
In \S\ref{sec:AHsym},
we provide a local-search based (almost) $r$-approximation algorithm for general rank $r$. With an observation of the connection between ah-symmetric (reflexive) generalized inverses and ha-symmetric (reflexive) generalized inverses, we can easily extend all the results in \S\ref{sec:AHsym} to the ha-symmetric case.
In \S\ref{sec:numres}, we present results of numerical experiments aimed at illustrating our results and confirming their applicability.
Finally, in \S\ref{sec:conc}, we make some brief concluding remarks.
In the Appendix, we demonstrate that the approximation ratios of
all of the local searches that we discuss are essentially tight.
Furthermore,  we investigate a more obvious local search
than the one we give (based directly on swaps seeking improvement in $\|H\|_1$),
and we establish some of its good and bad properties.

Before presenting our main results,
we note that it is useful to consider relaxing \ref{property2} completely, arriving at
 $\min\{\|H\|_1: \ref{property1}\}=\min\{\norm{H}_1: AHA=A\}$, which we re-cast as a linear-optimization problem \eqref{eqn:P} and its dual \eqref{eqn:D}:
\begin{equation*}\label{eqn:P}\tag{P}
\begin{array}{ll}
\mbox{minimize }& \langle J, H^+\rangle + \langle J, H^-\rangle\\
\mbox{subject to} & A(H^+ - H^-)A=A,\\
& H^+,H^-\ge 0;
\end{array}
\end{equation*}
\begin{equation*}\label{eqn:D}\tag{D}
\begin{array}{ll}
\mbox{maximize }& \langle A, W\rangle\\
\mbox{subject to} & -J\le A^\top W A^\top\le J.\\
\end{array}
\end{equation*}
More compactly, we can recast \eqref{eqn:D} as: $\max\{\langle A,W\rangle:~\|A^\top WA^\top\|_{\max}\le 1\}$.
In what follows, our approach is always to construct a feasible solution to \eqref{eqn:P} such that $H:=H^+ - H^-$
satisfies \ref{property2}, and measure the quality of the solution to \eqref{eqn:P}  against
a feasible solution that we construct for \eqref{eqn:D}.

\section{Symmetric results}\label{sec:sym}
We note that considerable effort has been made for tuning
hardware to efficiently handle sparse symmetric ``matrix-vector
multiplication'' (e.g., see \cite{Gkountouvas} and the references therein). Considering that virtually
any use of a generalized inverse $H$ would involve matrix-vector
multiplication, it can be very useful to prepare a sparse
symmetric generalized inverse $H$ from a symmetric $A$.

In this section, we assume that $A\in\mathbb{R}^{n\times n}$ is symmetric, and we seek to obtain an optimal solution to $\min\{\|H\|_1: \ref{property1}+\ref{property2},~ H^\top=H\}$.
Using \cite{FampaLee2018ORL}, we could first seek a $1$-norm minimizing reflexive generalized inverse $H$ of $A$ that is not necessarily symmetric.
If $H$ is not symmetric, then the natural symmetrization $(H+H^\top)/2$ is a symmetric generalized inverse with minimum $1$-norm, because doing this symmetrization cannot increase the convex function $\|\cdot\|_1$. However,  symmetrization is very likely to increase the rank and thus violate \ref{property2}.
Also, we next demonstrate that the extreme solutions of $\min\{\|H\|_1: \ref{property1},~ H^\top=H\}$ only have a guaranteed (sharp) bound of 
$r^2+r$
for the number of nonzeros, while the extreme solutions of $\min\{\|H\|_1: \ref{property1}\}$ have at most $r^2$ nonzeros.

\begin{proposition}~Suppose that $A\in\mathbb{R}^{n\times n}$ is symmetric and has rank $r$.
 \par \label{prop:basicguarantee_sym}
\begin{enumerate}[label=(\arabic*)]
\item Extreme solutions of the LP for $\min\{\|H\|_1: \ref{property1}\}$ have at most $r^2$ nonzeros. Furthermore, the bound  is sharp for all  $n\geq r\geq 1$.
     \label{prop:basic1}
\item Extreme solutions of the LP for $\min\{\|H\|_1: \ref{property1},~H^\top = H\}$ have  at most $r^2+r$ nonzeros. Furthermore, the bound is sharp 
    for $n-2\geq r \geq 3$.
    \label{prop:basic1sym}
\end{enumerate}
\end{proposition}
\begin{proof}
First, we claim that if $\rank(B)=p$, then the extreme solutions of
the LP associated with
\[
\min\{\|c \circ x\|_1:~ x\in\mathbb{R}^n,~ Bx=b\}
\]
have at most $p$ nonzeros, where $\circ$ is the element-wise product. By reformulating the problem as the LP
\[
\min\{|c|^\top (x^++x^-):~Bx^+ -Bx^- =b,~x^+,x^-\ge 0\},
\]
we see that the extreme solutions $(x^+,~x^-)$  have at least $2n-p$ zeros because there are only $p$ linearly-independent equations, which implies that $x:=x^+-x^-$ has at most $p$ nonzeros.

(1) Then we have
$\min\{\|H\|_1: \ref{property1}\}=\min\{\|\mathrm{vec}(H)\|_1: ~(A\otimes A)\mathrm{vec}(H)=\mathrm{vec}(A)\}$,
with $\rank(A\otimes A)=\rank(A)^2=r^2$.  To see that the bound is sharp, let $\hat{A}$ be a random $r\times r$ symmetric matrix (with iid entries taken from any absolutely continuous density), and then take
$A$ to be all zero except for $\hat{A}^+$ in the north-west corner.
Then with probability one: $\hat{A}$ is dense, $\hat{A}$ has rank $r$ (and then $\hat{A}^+=\hat{A}^{-1}$), and $A$ has a unique generalized inverse which is the M-P pseudoinverse $A^+$, which is all zero except for the dense $r\times r$ block $\hat{A}$ in the north-west corner. Thus \ref{prop:basic1} holds.

(2) $\min\{\|H\|_1: \ref{property1},~H^\top = H\}$ is equivalent to an LP on the variable $\mathrm{svec}(H):$
\begin{equation}\label{eqn:PS}\tag{\text{$P_{sym}$}}
\min\{\|\mathrm{svec}(J)\circ\mathrm{svec}(H)\|_1: ~(A\otimes_S A)\mathrm{svec}(H)=\mathrm{svec}(A)\},
\end{equation}
where $\otimes_S$ is the symmetric Kronecker product, and for any symmetric matrix $S$, $\mathrm{svec}(S)\in \mathbb{R}^{\frac12n(n+1)}$ is defined as
$$\mathrm{svec}(S) := (s_{11},\sqrt{2}s_{21},\cdots,\sqrt{2}s_{n1},s_{22},\sqrt{2}s_{32},\cdots,\sqrt{2}s_{n2},\cdots,s_{nn})^\top;$$
that is, we stack the columns of $S$ from the main diagonal downwards, but multiplying off-diagonal entries by $\sqrt{2}$
(see \cite{schacke2004kronecker} for details). By \cite[Theorem 3.6]{schacke2004kronecker}, we have that $A\otimes_S A$ has $\frac12 r(r+1)$ nonzero eigenvalues, thus $\rank(A\otimes_S A)=\frac12r(r+1)$. We know that $\mathrm{svec}(H)$ has at most $\frac{r^2+r}{2}$ nonzeros. Therefore, $H$ has at most $2$ times the nonzeros of $\mathrm{svec}(H)$, thus \ref{prop:basic1sym} holds.

Next we  construct a family of examples to show that the bound  is sharp for $n-2=r\geq 3$. Then for any $n-2\geq r$, we can take $A$ to be all zero except for an $(r+2)\times(r+2)$ block in the north-west corner. The dual of \eqref{eqn:PS} is
\begin{equation}\label{eqn:DS}\tag{\text{$D_{sym}$}}
    \max\{\mathrm{svec}(A)^\top \mathrm{svec}(W):~\|(A\otimes_S A) \mathrm{svec}(W)\|_{\max}\le \mathrm{svec}(J)\}.
\end{equation}
We could also view \eqref{eqn:DS} as $\max\{\langle A,W\rangle: ~ \|AWA\|_{\max}\le 1,~W^\top=W\}$.

Let
$X = \begin{bmatrix}
    \frac{r}{2(r-1)} & -\frac{r(r-2)}{2(r-1)^2}\\
    \frac{1}{2(r-1)}\mathbf{1}_{r-1} & \frac{r}{2(r-1)^2}\mathbf{1}_{r-1}
\end{bmatrix}\in\mathbb{R}^{r\times 2}$,
$Y = \begin{bmatrix}
    0 & -1\\
    -\mathbf{1}_{r-1} & \mathbf{0}_{r-1}
\end{bmatrix}\in\mathbb{R}^{r\times 2}$. Let $H_0=I_r - J_r$, $X^\top (H_0 + D) = Y^\top$, where $D$ is all zero except $D_{11} = \frac{r-1}{r}$.

Let $A_0 = (H_0+XY^\top + YX^\top)^{-1}$, and $A = \begin{bmatrix}I_r\\X^\top \end{bmatrix}A_0[I_r ~X]\in\mathbb{R}^{(r+2)\times(r+2)}$, $\rank(A)=\rank(A_0)=r$. Let $H = \begin{bmatrix}H_0 & Y\\ Y^\top & 0\end{bmatrix}$, and $W = \begin{bmatrix}A_0^{-1}(H_0+D)A_0^{-1} & 0\\ 0 & 0\end{bmatrix}$. These two symmetric matrices $H,W$ satisfy
\begin{align*}
    AHA &= \begin{bmatrix}I_r\\X^\top \end{bmatrix}A_0[I_r ~X]\begin{bmatrix}H_0 & Y\\ Y^\top & 0\end{bmatrix}\begin{bmatrix}I_r\\X^\top \end{bmatrix}A_0[I_r ~X]\\
        &=\begin{bmatrix}I_r\\X^\top \end{bmatrix}A_0(H_0+XY^\top + YX^\top)A_0[I_r ~X]
        =\begin{bmatrix}I_r\\X^\top \end{bmatrix}A_0[I_r ~X]=A,\\
    \langle A,W\rangle &= \mathrm{trace}((H_0+D)A_0^{-1})=\mathrm{trace}((H_0+D)(H_0+XY^\top + YX^\top))\\
    &=\mathrm{trace}(H_0^2+DH_0+2YY^\top)
    =\mathrm{trace}(H_0^2+2YY^\top)\\
    &=(r^2-r)+2r=\|H\|_1~,\\
    AWA &= \begin{bmatrix}I_r\\X^\top \end{bmatrix}A_0[I_r ~X]\begin{bmatrix}A_0^{-1}(H_0+D)A_0^{-1} & 0\\ 0 & 0\end{bmatrix}\begin{bmatrix}I_r\\X^\top \end{bmatrix}A_0[I_r ~X]\\
        &=\begin{bmatrix}I_r\\X^\top \end{bmatrix} (H_0+D)[I_r ~X]
        =\begin{bmatrix}H_0+D & Y\\ Y^\top & X^\top Y\end{bmatrix}\\
        \Rightarrow& ~\|AWA\|_{\max}\le 1 ~~\left(\text{because}~X^\top Y =\begin{bmatrix}
            -\frac{1}{2} & -\frac{r}{2(r-1)} \\
            -\frac{r}{2(r-1)} & \frac{r(r-2)}{2(r-1)^2}
        \end{bmatrix}~\text{and}~r\ge 3\right).
\end{align*}
Therefore, by  weak duality, $H$ and $W$ are  optimal solutions for  primal and dual, and $H$ has exactly $r^2+r$ nonzeros. Also, because $\mathrm{svec}(AWA)$ has exactly $\frac{r^2+r}{2}$ entries with value $\pm1$ corresponding to the position where $\mathrm{svec}(H)$ is nonzero, by complementary slackness, for any primal optimal solution $H^*$, $\mathrm{svec}(H^*)$ is zero in the positions where $\mathrm{svec}(H)$ is nonzero. Then we can easily solve the equation $(A\otimes_S A)\mathrm{svec}(H^*) =\mathrm{svec}(A)$ to obtain the unique solution $\mathrm{svec}(H)$. Therefore the primal problem has a unique optimal extreme solution $H$ with $r^2+r$ nonzeros.
\end{proof}

\begin{remark}
    For the case $n-1=r \geq 3$, we can only construct examples for which the unique optimal extreme solution of the LP for $\min\{\|H\|_1: P1,~H^\top = H\}$ has $r^2+r-1$ nonzeros.
\end{remark}

We seek to do better than what
\cref{prop:basicguarantee_sym}, part \ref{prop:basic1sym} provides.
 We want fewer nonzeros, and we want block structure.
To get these properties, we will give a new recipe for constructing a symmetric reflexive generalized inverse that has at most $r^2$ nonzeros. Our \emph{symmetric block construction} in the following theorem is the same block construction as from \cite{FampaLee2018ORL}, but only over the principal submatrices of $A$.

\begin{theorem}[the proof follows from \cite{FampaLee2018ORL}]\label{thm:symconstruction}
For a symmetric matrix $A\in\mathbb{R}^{n\times n}$, let $r := \rank(A)$. Let $\tilde{A}:=A[S]$ be any $r \times r$
nonsingular principal submatrix of $A$. Let $H\in\mathbb{R}^{n\times n}$ be equal to zero, except its submatrix with row/column indices $S$ is equal to $\tilde{A}^{-1}$. Then $H$ is a symmetric reflexive generalized inverse of $A$.
\end{theorem}

Letting $r:=\rank(A)$, when $r=1$ or $r=2$ and $A$ is nonnegative, construction of a $1$-norm minimizing symmetric reflexive generalized inverse can be based on the symmetric block construction over the
$r\times r$ principal submatrices of $A$, choosing one such that its inverse has minimum $1$-norm (see \url{https://arxiv.org/abs/1903.05744}).

Generally, when $\rank(A)\geq 2$, we cannot construct a $1$-norm minimizing symmetric reflexive generalized inverse based on the symmetric block construction.
For example, with
$$A:=\begin{bmatrix}
5& 4& 2\\
4& 5& -2\\
2& -2& 8
\end{bmatrix},$$
we have a symmetric reflexive generalized inverse
$H:=\frac{1}{81}A\quad (\text{because}~A^2=9A)$,
with $\|H\|_1=\frac{34}{81}$. While the three symmetric reflexive generalized inverses based on the symmetric block construction have $1$-norm equal to $\frac{17}{36},\frac{17}{36},2$, all greater than $\frac{34}{81}$.

For general $r:=\mathrm{rank}(A)$, we will efficiently find a symmetric reflexive generalized inverse following our symmetric block construction that is within  a factor of $r^2(1+\epsilon)$ of the  1-norm of the symmetric reflexive generalized inverse having minimum 1-norm. Before presenting the approximation result, we first establish  a useful lemma.
\begin{lemma}\label{lem:symm}
For a symmetric matrix $A\in\mathbb{R}^{n\times n}$, let $r:=\mathrm{rank}(A)$. Let $A[S]$ be a $r\times r$ nonsingular principal submatrix of $A$ with indices $S$, and let $A[T]$ be a principal submatrix obtained by swapping an element of $S$ with one from its complement. If $\abs{\det(A[T])}\le(1+\epsilon)\abs{\det(A[S])}$, then we have $\abs{\det(A[S,T])}\le\sqrt{(1+\epsilon)}\abs{\det(A[S])}$.
\end{lemma}
\begin{proof}
Without loss of generality, assume that $S=\{1,\dots,r\}$ and $T=\{1,\dots,r-1,r+1\}$. Then matrix $A$ is of the form
$$
\begin{bmatrix}
A[S] & a_{T} & *\\
a_{T}^\top & d & *\\
* & * & *
\end{bmatrix}.
$$
Because $A[S]$ is a nonsingular principal submatrix of $A$, the linear system $A[S]\cdot x = a_T$ has a unique solution $x$. $A[S,T]$ is obtained by replacing column $r$ of $A[S]$ with $a_T$, thus $\det(A[S,T]) = x_r \det(A[S])$.
On the other hand, because $\rank(A) = r=\rank(A[S])$, the Schur complement $d-a_T^\top A[S]^{-1} a_T = 0$, which implies that $d = x^\top a_T$. Therefore, $\det(A[T]) =x_r\det(A[S,T]) =  x_r^2\det(A[S])$. We have
\begin{displaymath}
\abs{\det(A[S,T])}^2=\abs{\det(A[S])}\abs{\det(A[T])}\le(1+\epsilon)\abs{\det(A[S])}^2.
\end{displaymath}
\end{proof}

\begin{definition}
Let $A$ be an arbitrary $n\times n$, rank-$r$ matrix. For $S$ an ordered subset of $r$ elements from $\{1,\dots,n\}$ and fixed $\epsilon\ge0$, if $|\det(A[S])|>0$ cannot be increased by a factor of more than $1+\epsilon$ by swapping an element of $S$ with one from its complement, then we say that $A[S]$ is a $(1+\epsilon)$-local maximizer for the absolute determinant on the set of $r\times r$ nonsingular principal submatrices of $A$.
\end{definition}

\begin{theorem}\label{thm:symapprox}
For a symmetric matrix $A\in\mathbb{R}^{n\times n}$, let $r:=\mathrm{rank}(A)$. Choose $\epsilon\ge0$, and let $\tilde{A}:=A[S]$ be a $(1+\epsilon)$-local maximizer for the absolute determinant on the set of $r\times r$ nonsingular principal submatrices of $A$. The $n\times n$ matrix $H$ constructed by \cref{thm:symconstruction} over $\tilde{A}$, is a symmetric reflexive generalized inverse (having at most $r^2$ nonzeros), satisfying $\|H\|_1\le r^2(1+\epsilon)\|H_{opt}^r\|_1$, where $H_{opt}^r$ is a $1$-norm minimizing symmetric reflexive generalized inverse of $A$.
\end{theorem}

\begin{proof}
We prove a stronger result $\|H\|_1\le r^2(1+\epsilon)\|H_{opt}\|_1$, where $H_{opt}$ is an optimal solution to \eqref{eqn:P}, which implies $\|H\|_1\le r^2(1+\epsilon)\|H_{opt}\|_1 \leq r^2(1+\epsilon)||H_{opt}^r||_1$.

Without loss of generality, we assume that $\tilde{A}$ is in the north-west corner of $A$. So we take $A$ to have the form $\begin{bmatrix}\tilde{A} & B\\ B^\top & D\end{bmatrix}$. Let $M=\mathrm{sign}(\tilde{A}^{-1})$, where $\mathrm{sign}(x)$ is defined as
$x/|x|$, if $x\ne0$, and 0 otherwise. Now we choose
\[
W:=\begin{bmatrix}
\tilde{W} & 0\\
0 & 0
\end{bmatrix}
:=\begin{bmatrix}
\tilde{A}^{-\top}M\tilde{A}^{-\top} & 0\\
0 & 0
\end{bmatrix}.
\]
The dual objective value
$\langle A,W\rangle =\mathrm{trace}(A^\top W)= \mathrm{trace}(M\tilde{A}^{-\top})=\|\tilde{A}^{-1}\|_1=\norm{H}_1$. Also,
$$
A^\top W A^\top=\begin{bmatrix}
M & M\tilde{A}^{-\top}B\\
B^\top\tilde{A}^{-\top}M & B^\top\tilde{A}^{-\top}M\tilde{A}^{-\top}B
\end{bmatrix}.
$$
Clearly $\norm{M}_{\max}\le 1$. Next, we consider $\bar{\gamma}:=M\tilde{A}^{-\top}\gamma=M\tilde{A}^{-1}\gamma$ ($\tilde{A}$ is symmetric), where $\gamma$ is an arbitrary column of $B$. By Cramer's rule, where $\tilde{A}_i(\gamma)$ is $\tilde{A}$ with column $i$ replaced by $\gamma$, we have
\begin{align*}
    \bar{\gamma} &= M\frac{1}{\det(\tilde{A})}\begin{bmatrix}\det(\tilde{A}_1(\gamma))\\ \vdots\\ \det(\tilde{A}_r(\gamma))\end{bmatrix}.
\end{align*}
And for $j=1,\dots, r$, using \cref{lem:symm}, we have
$$
\abs{\bar{\gamma}_j}=\sum_{i=1}^r\mathrm{sign}(\tilde{A}^{-1}_{ji})\frac{\det(\tilde{A}_i(\gamma))}{\det(\tilde{A})}\le\sum_{i=1}^r\frac{\abs{\det(\tilde{A}_i(\gamma))}}{\abs{\det(\tilde{A})}}\le r\sqrt{1+\epsilon},
$$
i.e., $\norm{M\tilde{A}^{-\top}B}_{\max}\le r\sqrt{1+\epsilon}$. Finally, we have
$$\norm{B^\top \tilde{A}^{-\top}M\tilde{A}^{-\top}B}_{\max}\le r^2\norm{B^\top\tilde{A}^{-1}}_{\max}\norm{\tilde{A}^{-1}B}_{\max}\le r^2(1+\epsilon).$$
Therefore, $\norm{A^\top WA^\top}_{\max}\le r^2(1+\epsilon)$; so then $\frac{1}{r^2(1+\epsilon)}W$ is dual feasible. By the weak duality for linear optimization, we have $\langle A,\frac{1}{r^2(1+\epsilon)}W\rangle=\frac{1}{r^2(1+\epsilon)}\norm{H}_1\le\norm{H_{opt}}_1$.
\end{proof}
\smallskip
\begin{remark}\label{rmk:epsilon}
In \cref{thm:symapprox}, we could have required the stronger condition that $\tilde{A}$ is a global maximizer for the absolute determinant on the set of $r\times r$ nonsingular principal submatrices of $A$. But we prefer our hypothesis, both because it is weaker and because we can find an $\tilde{A}$ satisfying our hypothesis by a simple finitely-terminating local search. Moreover, if $A$ is rational, and we choose $\epsilon$ positive and fixed, then our local search is efficient:
\end{remark}
\begin{theorem}\label{thm:symFPTAS}
    Let $A$ be rational. We have an FPTAS (fully polynomial-time approximation scheme; see \cite{WilliamsonShmoys}) for calculating a symmetric reflexive generalized inverse $H$ of $A$ that has $\|H\|_1$ within a factor of $r^2$ of $\|H_{opt}^r\|_1$, where $H_{opt}^r$ is a $1$-norm minimizing symmetric reflexive generalized inverse of $A$.
\end{theorem}
\begin{proof}
    Following the proof in \cite[Theorem 10]{FampaLee2018ORL}, we have that the local search reaches a $(1+\epsilon)$-local maximizer for the absolute determinant on the set of $r\times r$ nonsingular principal submatrices of $A$ in at most $\mathcal{O}(\mathrm{poly}(\mathrm{size}(A)))(1+\frac{1}{\epsilon})$ iterations, where $\mathrm{size}(A)$ is the number of bits in a binary encoding of $A$. Along with \cref{thm:symapprox}, we conclude that the local search is an FPTAS.
\end{proof}
\smallskip

\begin{remark}
The general idea of our proof follows the scheme of \cite[Theorem 9]{FampaLee2018ORL}
(the nonsymmetric situation). However, there is a mistake in the proof of \cite[Theorem 9]{FampaLee2018ORL}. To construct a dual feasible solution, \cite{FampaLee2018ORL} chose $\tilde{W}:=\tilde{A}^{-\top}(2I-J)\tilde{A}^{-\top}$ and claimed that $\langle A,W\rangle =\|H\|_1$. This claim does not generally hold for $r>2$, but by instead choosing $\tilde{W}:=\tilde{A}^{-\top} M\tilde{A}^{-\top}$ with $M=\mathrm{sign}(\tilde{A}^{-1})$ chosen as in our \cref{thm:symapprox}, \cite[Theorem 9]{FampaLee2018ORL} still holds as an $r^2(1+\epsilon)^2$-approximation algorithm.
\end{remark}

\section{ah-symmetric results}\label{sec:AHsym}

In this section, let $A$ be an arbitrary $m\times n$ real matrix. We seek to obtain a solution to $\min\{\|H\|_1: \ref{property1}+\ref{property2}+\ref{property3}\}$ (that is, a $1$-norm minimizing ah-symmetric reflexive generalized inverse). As we have mentioned, ah-symmetric
generalized inverses play a key role in solving least square problems. We develop an approximation
approach for
this problem that has many benefits, which we later summarize in \cref{fig:ahsym_compare}.

Note that if $H$ is an ah-symmetric generalized inverse, then $AH = AA^+$, where $A^+$ is the M-P pseudoinverse. Therefore, \ref{property2} ($HAH=H$) becomes a linear constraint $HAA^+=H$, which implies that $\min\{\|H\|_1: \ref{property1}+\ref{property2}+\ref{property3}\}$ can be cast as an LP. However, the extreme solutions of this LP only have a guaranteed bound of $mr + (m-r)(n-r)$ for the number of nonzeros, while the extreme solutions of $\min\{\|H\|_1: \ref{property1}+\ref{property3}\}$ have at most $mr$ nonzeros.

\begin{proposition}~\phantom{geojg}\par \label{prop:basicguarantee}
Suppose that $A\in \mathbb{R}^{m\times n}$ has rank $r$.
\begin{enumerate}[label=(\arabic*)]
\item Extreme solutions of the LP for $\min\{\|H\|_1: \ref{property1}+\ref{property3}\}$ have at most $mr$ nonzeros. Furthermore, the bound  is sharp for all  $m\geq n\geq r\geq 1$. \label{prop:basic13}
\item Extreme solutions of the LP for $\min\{\|H\|_1: \ref{property1}+\ref{property2}+\ref{property3}\}$ have at most $mr+(m-r)(n-r)$ nonzeros.
    \label{prop:basic123}
\end{enumerate}
\end{proposition}
\begin{proof}

We have $\min\{\|H\|_1: \ref{property1}+\ref{property3}\}=
\min\{\|H\|_1: ~AH=AA^+\}=$
\[
\min\{\|\mathrm{vec}(H)\|_1: ~(I_m\otimes A)\mathrm{vec}(H)=\mathrm{vec}(AA^+)\},
\]
with $\rank(I_m\otimes A)=\rank(I_m)\rank(A)=mr$.
To see that the bound is sharp, let $\hat{A}$ be a random dense $r \times m$  matrix  (with iid entries taken from any absolutely continuous density),
and then take
$A$ to be all zero except for $\hat{A}^+$ in the western $r$ columns.
Then with probability one: $\hat{A}$ is dense, $\hat{A}$ has rank $r$, and $A$ has a unique generalized inverse which is the M-P pseudoinverse $A^+$, which is all zero except for the dense $r\times m$ block $\hat{A}$ in the northern $r$ rows.
%
Thus \ref{prop:basic13} holds.

As for $\min\{\|H\|_1: \ref{property1}+\ref{property2}+\ref{property3}\}$, it can be written as
\[
\min\{\|\mathrm{vec}(H)\|_1: ~(I_m\otimes A)\mathrm{vec}(H)=\mathrm{vec}(AA^+), ~[(AA^+ \otimes I_n) - I_{mn}]\mathrm{vec}(H)=0\},
\]
with\hfil\break
\vskip-40pt
\begin{align*}
&\rank\left(\begin{bmatrix}I_m\otimes A\\(AA^+ \otimes I_n) - I_{mn}\end{bmatrix}\right)=\rank\left(\begin{bmatrix}I_m\otimes A\\(AA^+-I_m)\otimes I_n\end{bmatrix}\right)\\
=~&\rank\left(\begin{bmatrix}I_m\otimes A\\(AA^+-I_m)\otimes (I_n-A^+A)\end{bmatrix}\right)=mr+(m-r)(n-r).
\end{align*}
The second-to-last equation follows from the fact that $(AA^+-I_m)\otimes A^+A = ((AA^+-I_m)\otimes A^+)(I_m\otimes A)$.
Thus \ref{prop:basic123} holds.
\end{proof}

\begin{remark}\label{rem:badahsym}
with regard to \cref{prop:basicguarantee}, part \ref{prop:basic123}, the bound $mr+(m-r)(n-r)$ is sharp for $n=r^2$ and $r\ge 2$ (an example will be given in the Appendix). This implies that for large $n$, the best bound should be at least $mr+(r^2-r)(m-r)$. 
\end{remark}

We seek to do better than what
 \cref{prop:basicguarantee}, part \ref{prop:basic123} provides (and what \cref{rem:badahsym} indicates can actually be the case).
 We want fewer nonzeros, and we want block structure.
To get these properties,  we give a new
\emph{column block construction},
 producing an ah-symmetric reflexive generalized inverse that has at most $mr$ nonzeros.
\begin{theorem}\label{thm:ahconstruction}
For $A\in\mathbb{R}^{m\times n}$, let $r := \rank(A)$. For any $T$, an ordered subset of $r$ elements from $\{1,\dots,n\}$, let $\hat{A}:=A[:,T]$ be the $m \times r$ submatrix of $A$ formed by columns $T$. If $\rank(\hat{A})=r$, let
$
\hat{H} := \hat{A}^+ =(\hat{A}^\top\hat{A})^{-1}\hat{A}^\top.
$
The $n \times m$ matrix $H$ with all rows equal to zero, except rows $T$, which are given by $\hat{H}$, is an ah-symmetric reflexive generalized inverse of $A$.
\end{theorem}
\begin{proof}
Without loss of generality,
assume that $T:= (1,2,\dots,r)$, so we may write
\[
A=\left[\begin{array}{c}\hat{A}\;\;\; \hat{B}\end{array}\right], \;  H=\left[\begin{array}{c}\hat{H}\\0\end{array}\right].
\]

We have that $H$ satisfies:
\begin{itemize}
\item $\ref{property1}$, as
$
AHA=[\hat{A}\hat{H}\hat{A} \;\;\; \hat{A}\hat{H}\hat{B}]=[\hat{A} \;\;\; \hat{B}] = A,
$
where $\hat{A}\hat{H}\hat{A}= \hat{A}$ because $\hat{H}\hat{A} $ is the $r\times r$ identity matrix, and $\hat{A}\hat{H}\hat{B}= \hat{B}$ because, as $A$ (and  $\hat{A}$) has rank $r$, the columns of $\hat{B}$ are in the range of $\hat{A}$ and $\hat{A}\hat{H}$ is the projection matrix on the range of $\hat{A}$.
\item $\ref{property2}$,
as
\[
HAH=\left[\begin{array}{c}\hat{H}\hat{A}\hat{H} \\0\end{array}\right]=\left[\begin{array}{c}\hat{H}\\0\end{array}\right]=H,
\]
where we again use  the fact that  $\hat{H}\hat{A} $ is the $r\times r$ identity matrix.
\item $\ref{property3}$, as
$
AH = \hat{A}\hat{H} = \hat{A}(\hat{A}^\top\hat{A})^{-1}\hat{A}^\top
$
is symmetric.
\end{itemize}
\end{proof}

\begin{remark}
We have already mentioned that if $H$ is an ah-symmetric generalized inverse, then $AH = AA^+$. Therefore, \ref{property2} ($HAH=H$) becomes a linear constraint $HAA^+=H$. In fact, rather than linearize using the M-P pseudoinverse $A^+$, we can take any column block ah-symmetric generalized inverse $\hat{H}$ of $A$, and linearize
more efficiently via $HA\hat{H}=H$. The cost of calculating such an $\hat{H}$ is the cost of
calculating the M-P pseudoinverse of an $m\times r$ matrix, rather than the M-P pseudoinverse of the $m\times n$ matrix $A$.
\end{remark}

Similarly as before, we note that it is useful to consider relaxing \ref{property2}, arriving at $\min\{\|H\|_1: \ref{property1}+\ref{property3}\}=\min\{\norm{H}_1: AHA=A,~(AH)^\top=AH\}$, which we re-cast as a linear-optimization problem \eqref{eqn:Pah} and its dual \eqref{eqn:Dah}:
\begin{equation*}\label{eqn:Pah}\tag{\text{$P_{ah}$}}
\begin{array}{ll}
\mbox{minimize }& \langle J, H^+\rangle + \langle J, H^-\rangle\\
\mbox{subject to} & A(H^+ - H^-)A=A,\\
& (H^+ - H^-)^\top A^\top = A(H^+ - H^-),\\
& H^+,H^-\ge 0.
\end{array}
\end{equation*}
\begin{equation*}\label{eqn:Dah}\tag{\text{$D_{ah}$}}
\begin{array}{ll}
\mbox{maximize }& \langle A, W\rangle\\
\mbox{subject to} & -J\le A^\top W A^\top + A^\top (V^\top-V)\le J\\
\end{array}
\end{equation*}
We can see \eqref{eqn:Dah}   as: $\max\{\langle A,W\rangle:~\|A^\top WA^\top+A^\top U\|_{\max}\le 1, ~U^\top=-U\}$.


When $\rank(A) = 1$, construction of a $1$-norm minimizing ah-symmetric reflexive generalized inverse can be based on the column block construction
over a column $\hat{a}$ that minimizes $\|\hat{a}^+\|_1$  (see \url{https://arxiv.org/abs/1903.05744}).

{\color{red} \subsection{Rank 2}}
Generally, when $\rank(A)=2$, we cannot construct a $1$-norm minimizing ah-symmetric reflexive generalized inverse based on the column block construction. Even under the condition that $A$ is nonnegative, we have the following example:
$$
A =\begin{bmatrix}1&3&8\\ 2&2&8\\ 3&1&8\end{bmatrix}.
$$
Note that $\mathrm{rank}(A)=2$ because $a_3=2a_1+2a_2$. We have an ah-symmetric reflexive generalize inverse with $1$-norm $\frac98$,
$$
H :=\begin{bmatrix}
-\frac14 & 0 &\frac14 \\
\frac14 & 0 & -\frac14\\
\frac{1}{24} &\frac{1}{24} & \frac{1}{24}
\end{bmatrix}.
$$
However, the three ah-symmetric reflexive generalized inverses based on our column block construction have $1$-norm
$\frac{31}{24}, \frac{31}{24}, \frac{7}{6}$, respectively.
%
Nevertheless, under an efficiently-checkable technical condition, when $\rank(A) = 2$, construction of a $1$-norm minimizing ah-symmetric reflexive generalized inverse can be based on the column block construction (see \url{https://arxiv.org/abs/1903.05744}).

{\color{red}\subsection{Approximation}}
For general $r:=\mathrm{rank}(A)$, we will efficiently find an ah-symmetric reflexive generalized inverse following our column block construction that is within  a factor $r(1+\epsilon)$ of the 1-norm of the ah-symmetric reflexive generalized inverse having minimum 1-norm.
\smallskip
\begin{definition}
Let $A$ be an arbitrary $m\times n$, rank-$r$ matrix, and let $S$ be an ordered subset of $r$ elements from $\{1,\dots,m\}$ such that these $r$ rows of $A$ are linearly independent. For $T$ an ordered subset of $r$ elements from $\{1,\dots,n\}$, and fixed $\epsilon\ge0$, if $|\det(A[S,T])|$ cannot be increased by a factor of more than $1+\epsilon$ by swapping an element of $T$ with one from its complement, then we say that $A[S,T]$ is a $(1+\epsilon)$-local maximizer for the absolute determinant on the set of $r\times r$ nonsingular submatrices of $A[S,:]$.
\end{definition}

\begin{lemma}\label{thm:WU}
Let $T$ be an ordered subset of $r$ elements from $\{1,\dots,n\}$ and $\hat{A}:=A[:,T]$ be the $m\times r$ submatrix of an $m\times n$ matrix $A$ formed by columns $T$, and $\rank(\hat{A})=r$. There exists an $m\times n$ matrix $W$ and a skew-symmetric $m\times m$ matrix $U$ such that
\begin{equation*}
\hat{A}^\top W A^\top + \hat{A}^\top U=E,
\end{equation*}
where $E :=\mbox{sign}(\hat{A}^+)$. Furthermore, $\langle A,W\rangle=\|\hat{A}^+\|_1$.
\end{lemma}
\begin{proof}
Suppose that $\tilde{A}:=A[S,T]$ is the nonsingular $r\times r$ submatrix of $\hat{A}$ formed by rows $S:=\{i_1, i_2,\dots,i_r\}$. Let $\hat{W}$ be a $r\times r$ matrix and $W$ be an $m\times n$ matrix with all elements equal to zero, except the ones in rows $S$ and columns $T$, which are given by the respective elements in $\hat{W}$.
If we choose $\hat{W}$ and $U$ to be
$$
\hat{W}:=\tilde{A}^{-\top}E\hat{A}(\hat{A}^\top\hat{A})^{-1}=\tilde{A}^{-\top}E(\hat{A}^\top)^+
$$
and
$$
U := \hat{A} \hat{W}^\top D -D^\top\hat{W} \hat{A}^\top +D^\top \tilde{A}^{-\top} E- E^\top \tilde{A}^{-1}D~,
$$
where $D$ is a $r\times m$ matrix with all elements equal to zero, except $D_{1i_1}=D_{2i_2}=\dots=D_{ri_r}=1$.

Because $D\hat{A}=\tilde{A}$, we have
\begin{align*}
\hat{A}^\top U &= \hat{A}^\top\hat{A}\hat{W}^\top D-\tilde{A}^\top\hat{W}\hat{A}^\top+E - \hat{A}^\top E^\top\tilde{A}^{-1}D\\
&=E-\tilde{A}^\top\hat{W}\hat{A}^\top + (\hat{W}(\hat{A}^\top\hat{A})-\tilde{A}^{-\top}E\hat{A})^\top D\\
&=E-\tilde{A}^\top\hat{W}\hat{A}^\top.
\end{align*}
Hence, $\hat{A}^\top W A^\top +\hat{A}^\top U =\tilde{A}^\top \hat{W} \hat{A}^\top +\hat{A}^\top U =E$. Furthermore,
\begin{displaymath}
\langle A, W\rangle=\mbox{trace}(\tilde{A}^\top \hat{W}) = \mbox{trace}(E(\hat{A}^\top )^+ ) = \langle \hat{A}^+,E\rangle = \|\hat{A}^+\|_1~.
\end{displaymath}
\end{proof}

\begin{theorem}
\label{thmrwithP3}
Let $A$ be an arbitrary $m\times n$, rank-$r$ matrix, and let $S$ be an ordered subset of $r$ elements from $\{1,\dots,m\}$ such that these $r$ rows of $A$ are linearly independent. Choose $\epsilon\ge0$, and let $\tilde{A}:=A[S,T]$ be a $(1+\epsilon)$-local maximizer for the absolute determinant on the set of $r\times r$ nonsingular submatrices of $A[S,:]$. Then the $n \times m$ matrix $H$ constructed by \cref{thm:ahconstruction} over $\hat{A}:=A[:,T]$, is an ah-symmetric reflexive generalized inverse of $A$ satisfying $\|H\|_1\le r(1+\epsilon)\|H_{opt}^{r}\|_1$, where $H_{opt}^r$ is a $1$-norm minimizing ah-symmetric reflexive generalized inverse of $A$.
\end{theorem}
\begin{proof}
We prove a stronger result $\|H\|_1\le r(1+\epsilon)\|H_{opt}^{ah}\|_1$, where $H_{opt}^{ah}$ is an optimal solution of \eqref{eqn:Pah}, which implies $\|H\|_1\le r(1+\epsilon)\|H_{opt}^{ah}\|_1\le r(1+\epsilon)\|H_{opt}^{r}\|$. We will construct a dual feasible solution with objective value $\frac{1}{r(1+\epsilon)}\| H\|_1$. By weak duality for linear optimization, we will then have $\frac{1}{r(1+\epsilon)}\| H\|_1\le \|H_{opt}^{ah}\|_1$.

By \cref{thm:WU}, we can choose $W$ and a skew-symmetric matrix $U$ such that
$
\hat{A}^\top W A^\top + \hat{A}^\top U=E~.
$
and
$
\langle A,W\rangle = \|\hat{A}^+\|_1=\|H\|_1~.
$

So it is sufficient to demonstrate that $\|A^\top W A^\top +A^\top U\|_{\max}\le r(1+\epsilon)$, then $\frac{1}{r(1+\epsilon)}W, \frac{1}{r(1+\epsilon)}U$ is dual feasible and $\langle A,\frac{1}{r(1+\epsilon)}W \rangle=\frac{1}{r(1+\epsilon)}\|H\|_1$.

First, it is clear that
$
\| \hat{A}^\top W A^\top +  \hat{A}^\top U\|_{\max} =\|E\|_{\max}= 1\leq r(1+\epsilon).
$
Next, we consider any column $\hat{b}$ of $\hat{B}$, because $\mbox{rank}(\hat{A})=r=\mbox{rank}(A)$, we know that $\hat{b}=\hat{A}\beta$, $\beta\in\mathbb{R}^r$, which implies $\tilde{b}=\tilde{A}\beta$. By Cramer's rule, where $\tilde{A}_i(\tilde{b})$ is $\tilde{A}$ with column $i$ replaced by $\tilde{b}$, we have
$$
|\beta_i|=\frac{|\det(\tilde{A}_i(\tilde{b}))|}{|\det(\tilde{A})|}\le 1+\epsilon,
$$
because $\tilde{A}$ is a $(1+\epsilon)$-local maximizer for the absolute determinant of $A[S,:]$. Therefore
\begin{align*}
\|\hat{b}^\top W A^\top +  \hat{b}^\top U \|_{\max}&= \|\beta^\top (\hat{A}^\top W A^\top +\hat{A}^\top U)\|_{\max} \\
&= \|\beta^\top E\|_{\max}\le\sum_{i=1}^r|\beta_i|\le r(1+\epsilon).
\end{align*}
\end{proof}

\smallskip
\begin{remark}
In \cref{thmrwithP3}, we could have required the stronger condition that $\tilde{A}$ is a global maximizer for the absolute determinant on the set of $r\times r$ nonsingular submatrices of $A_{\sigma}$. But we prefer our hypothesis --- the reasons are the same as in \cref{rmk:epsilon}. And the local search is efficient:
\end{remark}
\begin{theorem}\label{thm:ahFPTAS}
    Let $A$ be rational. We have an FPTAS for calculating an ah-symmetric reflexive generalized inverse $H$ of $A$ that has $\|H\|_1$ within a factor of $r$ of $\|H_{opt}^{r}\|_1$, where $H_{opt}^r$ is a $1$-norm minimizing ah-symmetric reflexive generalized inverse of $A$.
\end{theorem}

As we have mentioned, ah-symmetric generalized inverses
have the key use for solving least-squares problems.
In \cref{fig:ahsym_compare}, we compare various possibilities for
calculating ah-symmetric generalized inverses, highlighting
the excellent properties of the solution produced by our local search.

Considering \cref{fig:ahsym_compare},
we dismiss methods based on minimizing the 0-norm as
we do not have nice computational methods for them, and
they suffer from not being able to control the magnitudes of
entries. Concentrating now on tractable optimization methods
(that seek to keep the magnitude of entries under control),
we have LP-based methods and our local search.

\begin{figure}[t]
\caption{Comparing options for ah-symmetric generalized inverses}\label{fig:ahsym_compare}
\begin{center}
\begin{mpsupertabular}{c|c|c|c|c|c|c|c||l}
  \rot{arbitrary ah-sym} & \rot{arbitrary reflexive ah-sym} & \rot{0-norm min ah-sym} & \rot{0-norm min reflexive ah-sym} & \rot{LP:\ref{property1}+\ref{property3}} & \rot{LP:\ref{property1}+\ref{property2}+\ref{property3}} & \rot{arbitrary block} & \rot{our local search} & \rot{} \\
    \hline
  \xmark & \xmark & \xmark & \xmark & \cmark  &  \cmark  & \xmark &  \cmark  & entries under control\footnote{via 1-norm pressure} \\
  \xmark & \cmark & \xmark & \cmark & \xmark &  \cmark  &\cmark
   &  \cmark  & guaranteed low rank ($=r$\footnote{via \ref{property2}}) \\
  \xmark & \xmark & \xmark & \xmark & \xmark &  \xmark  &\cmark &  \cmark  & structured\footnote{via column block construction}\\
  \xmark & \xmark & \cmark & \cmark & \cmark\footnote{see \cref{prop:basicguarantee}, part \ref{prop:basic13}} &  \xmark\footnote{no more than $rm + (m-r)(n-r)$ nonzeros: see \cref{prop:basicguarantee}, part \ref{prop:basic123}}  &\cmark
   &  \cmark  & guaranteed sparsity ($\leq rm$ nonzeros\footnote{via column block construction}) \\
  \xmark & \xmark & \cmark & \cmark & \cmark
   &  \cmark  &\xmark &  \cmark  & induced sparsity\footnote{via 1-norm or 0-norm pressure} \\
    \cmark & \cmark & \xmark\footnote{see \cref{prop:hard}, \ref{sginv13}} & \xmark\footnote{see \cref{prop:hard}, \ref{sginv123}} & \cmark &  \cmark  &\cmark &  \cmark  & calculate efficiently\\
\end{mpsupertabular}
\end{center}
\end{figure}

We can see some very important advantages of our local search:
 (i) comparing just to LP-based methods,
our local search has (block) structure, while the LP-based methods have
no guaranteed structure; (ii) comparing further to the LP based
on \ref{property1}+\ref{property3}, our local search has a low-rank
guarantee, while the LP method does not. (iii) instead comparing further to the LP based
on \ref{property1}+\ref{property2}+\ref{property3},
our local search has a much better
sparsity guarantee than the LP method.


\begin{remark}
    $H$ is a ha-symmetric (reflexive) generalized inverse of $A$ if and only if $H^\top$ is an ah-symmetric (reflexive) generalized inverse of $A^\top$. Following this observation, we can extend all the results in section \S\ref{sec:AHsym} to the ha-symmetric case.
\end{remark}

\section{Numerical experiments}
\label{sec:numres}

Next, we report on some numerical results to illustrate and confirm the applicability of our  proposed approach for constructing generalized inverses. For that, we have selected the ah-symmetric case and implemented a local-search algorithm based on Theorems \ref{thm:ahconstruction} and \ref{thmrwithP3}.  For the purpose of  computations, the parameter $\epsilon$ in Theorem \ref{thmrwithP3} was chosen to be zero.

The algorithm was coded in Matlab R2018a, and
to evaluate its performance, we also solved the linear programs LP:\ref{property1}+\ref{property3} and LP:\ref{property1}+\ref{property2}+\ref{property3} for the smaller instances,
 with Gurobi v.9.0.2. We ran our experiments on a
16-core machine (running Windows Server 2016 Standard):
two Intel Xeon CPU E5-2667 v4 processors
running at 3.20GHz, with 8 cores each, and 128 GB of memory.

The local-search algorithm implemented selects an $m\times r$ rank-$r$ submatrix of a given matrix $A$ and constructs a reflexive ah-symmetric  generalized inverse of $A$, as described in Theorem  \ref{thm:ahconstruction}. Our test matrices were randomly generated with  varied dimensions and ranks. We used the Matlab function \emph{sprand}, which  generates a random  $m\times n$ dimensional matrix $A$ with singular values
given by a nonnegative input vector $rc$.
We generated dense matrices and selected the $r$ nonzeros of $rc$
as the decreasing vector $M\times(\rho^{1},\rho^{2},\ldots,\rho^{r})$, where $M=2$, and $\rho=(1/M)^{(2/(r+1))}$.


Average results for our first experiment are reported in Table \ref{table1}. We solved  LP:\ref{property1}+\ref{property3} and LP:\ref{property1}+\ref{property2}+\ref{property3} for 5 instances of each dimension/rank indicated in the first column of the table, limiting the computational time to solve each instance to 2 hours (i.e., 7200 seconds). Our purpose is to demonstrate how fast the time  to solve these problems increases as we increase the dimension/rank of our test matrices. In the third column of Table \ref{table1}, we give  the number of instances solved to optimality within the time limit. The  average times in the second column, only take into account  the instances solved to optimality. We note that  for $m,n,r=200,100,50$, we could only solve one  instance with each LP model. The results demonstrate that computing ah-symmetric generalized inverses by solving the LP problems does not scale well and is not a practical approach for instances of moderate size, even when the reflexive property \ref{property2} is not imposed.
\begin{table}[!ht]
\centering{
	\begin{tabular}{r|rr|cc}
		\hline
		\multicolumn{1}{c|}{} & \multicolumn{2}{c|}{Time (sec)} & \multicolumn{2}{c}{Instances solved}
		\\
		\multicolumn{1}{c|}{$m, n, r$}                         & LP:\ref{property1}+\ref{property3}           & LP:\ref{property1}+\ref{property2}+\ref{property3}          & LP:\ref{property1}+\ref{property3}             & LP:\ref{property1}+\ref{property2}+\ref{property3}             \\ \hline
		40, 20, 10                                    & 1.76           & 1.98           & 5                & 5                 \\
		80, 40, 20                                    & 41.39          & 40.19          & 5                & 5                 \\
		120, 60, 30                                   & 384.34         & 390.34         & 5                & 5                 \\
		160, 80, 40                                   & 4130.99        & 4248.34        & 4                & 3                 \\
			\multicolumn{1}{c|}{200, 100, 50}                                  & 4197.86        & 4707.34        & 1                & 1                 \\ \hline
	\end{tabular}
	\caption{Computation of minimum 1-norm $H$ with LP models \label{table1}}
}
\end{table}

In Table \ref{table2}, we compare  the optimal solution of LP:\ref{property1}+\ref{property2}+\ref{property3}   to the reflexive ah-symmetric generalized inverse obtained by the local search, showing the 1-norm ($\|H\|_1$) and  sparsity ($\|H\|_0$, computed with tolerance $10^{-6}$). In this experiment we use 30 instances of each dimension/rank indicated in the first column of the table, and report the mean and standard deviation (in parenthesis) of the norms for each group. The results confirm the advantage of the local search over the LP solution in obtaining sparser matrices (via our column block construction), while keeping the magnitude
of the entries reasonably small (via our approximate 1-norm minimization).
\begin{table}[!ht]
	\centering{
	\begin{tabular}{r|rr|rr}
		\hline
		\multicolumn{1}{c|}{} & \multicolumn{2}{c|}{$\|H\|_1$}                                & \multicolumn{2}{c}{$\|H\|_0$}\\ 
		\multicolumn{1}{c|}{$m, n, r$}                         & \multicolumn{1}{c}{Local Search} & \multicolumn{1}{c|}{LP:\ref{property1}+\ref{property2}+\ref{property3} } & \multicolumn{1}{c}{Local Search}    & \multicolumn{1}{c}{LP:\ref{property1}+\ref{property2}+\ref{property3} }   \\ \hline
		40, 20, 10                                    & 62.73 (\hspace{0.07in}6.04)                     & 54.54 (\hspace{0.07in}3.91)              & 387.43 (\hspace{0.1in}8.24)                       & 547.73 (\hspace{0.07in}45.58)             \\
		80, 40, 20                                    & 191.22 (24.58)                   & 147.68 (11.58)            & 1547.37 (17.54)                     & 2373.53 (\hspace{0.07in}97.02)            \\
		120, 60, 30                                   & 354.39 (36.88)                   & 263.53 (15.53)            & 3489.70 (38.51)                        & 5434.53 (192.41)           \\ \hline
	\end{tabular}
\caption{Comparison between  local search and LP (Mean(Std Dev)) \label{table2}}
}
\end{table}

In Tables \ref{table3} and \ref{table4} we investigate the performance of the local search.  In the second column of these tables we show the relative decrease on the 1-norm of the reflexive ah-symmetric generalized inverse, comparing the solution $H$ obtained by the local search to the matrix $H^0$ used to initialize the algorithm.  We apply a phase-one local-search algorithm to construct  $H^0$.
In the two last columns of the tables we report the total computational time and number of column swaps performed by the local search. The time to compute the initial matrix $H^0$ and to perform the local search are both included.

In Table \ref{table3}, we consider 30 instances of each dimension/rank, and we present the mean and standard deviation for each group. We note that the local search is effective in reducing the 1-norm of the initial matrix and is much faster than solving LP problems of smaller dimensions, as can be observed from the results in Table \ref{table1}. The average number of column swaps and the standard deviation for the norm decrease demonstrates that the algorithm is very stable, converging to similar solutions after swapping about 60\%   of the columns in the matrix.

\begin{table}[!ht]
		\centering{
	\begin{tabular}{c|c|c|r}
		\hline
		\multicolumn{1}{c|}{$m, n, r$} & \multicolumn{1}{c|}{ $\frac{\|H^0\|_1 - \|H\|_1}{\|H^0\|_1}$} & \multicolumn{1}{c|}{Time (sec)} & \multicolumn{1}{c}{Swaps} \\ \hline
		250, 125, 25                 &0.90 (0.08) & 0.03 (0.01)  & 73.03 (\hspace{0.07in}11.11)          \\
		500, 250, 50                 &0.94 (0.06)   & 0.10 (0.03)  & 159.67 (\hspace{0.07in}20.84)           \\
		1000, 500, 100               &0.91 (0.13)   & 0.84 (0.29) & 293.33 (106.47)          \\ \hline
	\end{tabular}
\caption{Performance of the local search - medium-size instances (30 of each dimension)(Mean(Std Dev)) \label{table3}}
}
\end{table}

In Table \ref{table4}, we consider 5 instances of each dimension/rank, and present average results. Our purpose with this last experiment is to show  the scalability of the local search. The algorithm  is able to construct sparse reflexive ah-symmetric generalized inverses for our test matrices with up to 10000 rows, 1000 columns and rank 100,   in less than 1.1 second on average.
\begin{table}[!ht]
	\centering{
	\begin{tabular}{c|c|c|c}
		\hline
		\multicolumn{1}{c|}{$m, n, r$} & \multicolumn{1}{c|}{ $\frac{\|H^0\|_1 - \|H\|_1}{\|H^0\|_1}$} & \multicolumn{1}{c|}{Time (sec)} & \multicolumn{1}{c}{Swaps} \\ \hline
		5000, 500, 50                & 0.91                             & 0.21                          & 121.8                     \\
		7500, 750, 75                & 0.89                             & 0.65                          & 172.4                    \\
		10000, 1000, 100             & 0.89                             & 1.09                         & 204.8                    \\ \hline
	\end{tabular}
\caption{Performance of the local search - large-size instances (5 of each dimension)(Mean) \label{table4}}
}
\end{table}

\section{Conclusions and open questions} \label{sec:conc}

Generalized inverses have a wide variety of uses in matrix algebra and its applications. Sparsity of a generalized inverse  is highly preferred for efficiency in its use;
structured sparsity and low rank (=reflexivity) are both preferred for explainability.
(Approximate) 1-norm minimization is useful for keeping entries under control and for inducing sparsity.

When the input matrix is symmetric, a symmetric generalized inverse is useful in making matrix algebra more efficient.
Ah-symmetric (resp., ha-symmetric) generalized inverses have the key use in solving least-squares
(resp., minimum-norm) problems. Reflexive generalized inverses have low rank (same as the
input matrix), and this is usually preferred in applications.

We have given local-search algorithms that efficiently produce: (i) symmetric reflexive generalized inverses of symmetric matrices, (ii) reflexive ah-symmetric (ha-symmetric) generalized inverses
. Our algorithms produce generalized inverses
with guaranteed structured sparsity, with low rank (same as the input matrix), and with entries under control (by approximate 1-norm minimization). No other known methods have all of these nice properties.

Of course giving efficient algorithms to improve any of our approximation ratios is a nice challenge. Even for special classes of matrices, this could be interesting. It would be nice to
resolve the complexity of
$\min\{\|H\|_0:~\ref{property1}+\ref{property2},~H^\top=H\}$ and
$\min\{\|H\|_0:~\ref{property1}+\ref{property3}+\ref{property4}\}$.
Finally, with respect to the results in \S\S\ref{sec:r+1}--\ref{sec:1norm_ex},
we would like to understand the behavior of 1-norm based local search  for $r+1<n<2r$.


%
%
%



\bibliographystyle{siamplain}
\bibliography{ginv}

\section{Appendix}\label{sec:examples}

In this appendix, first we present several results with regard to the approximation ratio of local-search approximation algorithms based on the determinant (i.e., the algorithms of \cref{thm:symapprox,thmrwithP3}
, and \cite[Theorem 9]{FampaLee2018ORL}). Although in most numerical tests the achieved approximation ratio is less  than $2$ (see \cite{FLPX}, forthcoming), we demonstrate (in \S\ref{sec:det_ex}) that the approximation ratios for the local searches based on the determinant (as in our theorems) are best possible. We also consider  local-search algorithm based on the actual objective function, i.e., the 1-norm of the inverse. We demonstrate (in \S\ref{sec:r+1}) that for
rank-$r$ $r\times (r+1)$ matrices, the approximation ratio is $\frac{2r}{r+1}<2$, while there is no constant approximation ratio of the local search based on the 1-norm of the inverse, for
$r\times n$ matrices when $n\geq 2r$ (see \S\ref{sec:1norm_ex}).

Finally, we give a family of examples demonstrating that the bound in \cref{prop:basicguarantee} part \ref{prop:basic123} is sharp for $n=r^2$ and $r\ge 2$ (see \S\ref{sec:sb123_ex}), and so in fact there are LP solutions that are much worse than what our column block solution provides
(i.e., $mr+(r^2-r)(m-r)$ nonzeros vs. $mr$ nonzeros).

\subsection{Worst case for local search based on the determinant} \label{sec:det_ex}
We present examples to demonstrate that the approximation ratios for the local search based on the determinant are essentially best possible. We will first give a $r\times r$ nonsingular matrix $\tilde{A}$, then construct a rank-$r$ matrix $A$ that has a local-maximizer $\tilde{A}$ but has another block $B$ with $\|B\|_1$ close to $\|A\|_1$ divided by the approximation ratio.
\begin{example}
    Let $\tilde{A}^{-1}$ be a $r\times r$ Toeplitz matrix, $\delta_L,\delta_U\ge 0$ and small, $ \tilde{A}^{-1}:=$
    $$
   \begin{bmatrix}
    1 & 1+\delta_U & 1+2\delta_U & \ddots & 1+(r-2)\delta_U & 1+(r-1)\delta_U\\
    1+\delta_L & 1 & 1+\delta_U & 1+2\delta_U & \ddots & 1+(r-2)\delta_U\\
    1+2\delta_L & 1+\delta_L & 1 & 1+\delta_U &\ddots & \ddots\\
    \ddots & \ddots & \ddots & \ddots &\ddots & \ddots\\
    1+(r-2)\delta_L & \ddots & \ddots & \ddots &\ddots & 1+\delta_U\\
    1+(r-1)\delta_L & 1+(r-2)\delta_L & \ddots & \ddots &1+\delta_L & 1
    \end{bmatrix}.
    $$
    If $\delta_L=\delta_U$, then $\tilde{A}$ is symmetric. Note that
    $\mathrm{rank}(\tilde{A}^{-1})=r$ when $\delta_L,\delta_U$ are not both $0$. This is because by several subtractions of two rows or two columns, $\tilde{A}^{-1}$ has the same determinant as

    $$
    \begin{bmatrix}
    1 & 0 & 0 & \ddots & 0 & \delta_U\\
    \delta_L & -(\delta_L+\delta_U) & 0 & 0 & \ddots & 0\\
    2\delta_L & 0 & -(\delta_L+\delta_U) & 0 &\ddots & \ddots\\
    \ddots & 0 & 0 & \ddots &\ddots & \ddots\\
    (r-2)\delta_L & \ddots & \ddots & \ddots &\ddots & 0\\
    (r-1)\delta_L & 0 & \ddots & \ddots &0 & -(\delta_L+\delta_U)\\
    \end{bmatrix}
    $$
    which implies $
    \det(\tilde{A}^{-1})= [-(\delta_L+\delta_U)]^{r-1}-(r-1)\delta_L\delta_U[-(\delta_L+\delta_U)]^{r-2}
    $. Now we construct the rank-$r$ $m\times n$ matrices as following:
    \begin{enumerate}[label=(\arabic*)]
        \item For \cite[Theorem 9]{FampaLee2018ORL}, we construct
        $$
        A = \begin{bmatrix}
        \tilde{A} & b & 0\\
        c^\top & d & 0\\
        0 & 0 & 0
        \end{bmatrix};
        $$
        \item For \cref{thm:symapprox}, we choose $\delta_L=\delta_U$ and construct
        $$
        A = \begin{bmatrix}
        \tilde{A} & b & 0\\
        b^\top & d & 0\\
        0 & 0 & 0
        \end{bmatrix};
        $$
        \item For \cref{thmrwithP3}, we construct
        $$
        A = \begin{bmatrix}
        \tilde{A} & b & 0\\
        0 & 0 & 0
        \end{bmatrix},
        $$
    \end{enumerate}
    where $b = \tilde{A}\mathbf{1}$, $c^\top = \mathbf{1}^\top \tilde{A}$, $d=\mathbf{1}^\top \tilde{A}\mathbf{1}$. If $A$ is symmetric, then $c^\top = b^\top$.
    In all cases, $\tilde{A}$ is clearly a local maximizer because the determinant does not change when swap $b$ (resp. $c$) with any column (resp. row) of $\tilde{A}$, and
    $$
    \|\tilde{A}^{-1}\|_1= r^2 + \frac{r^3-r}{6}(\delta_L+\delta_U).
    $$
    Now, we compute the 1-norm of the reflexive generalized inverse, when we swap $b$ with column $1$ of $\tilde{A}$, i.e., compute $\|(\tilde{A}_1(b))^{-1}\|_1$. Let $e_1$ be the unit vector with $1$ in the first entry and $0$ otherwise, and let $\tilde{a}_j$ be the $j$th column of $\tilde{A}$, then $\tilde{A}_1(b)=\tilde{A} + (b-\tilde{a}_1)e_1^\top$. By the Sherman-Morrison formula, we have
    \begin{align*}
    (\tilde{A}_1(b))^{-1}& = \tilde{A}^{-1} - \frac{\tilde{A}^{-1}(b-\tilde{a}_1)e_1^\top \tilde{A}^{-1}}{1+e_1^\top \tilde{A}^{-1}(b-\tilde{a}_1)}
    = \tilde{A}^{-1} - \frac{(\mathbf{1}-e_1)e_1^\top \tilde{A}^{-1}}{1+e_1^\top (\mathbf{1}-e_1)}
    = \tilde{A}^{-1} - (\mathbf{1}-e_1)e_1^\top \tilde{A}^{-1}.
    \end{align*}
    Thus $(\tilde{A}_1(b))^{-1}=$
    $$
    \begin{bmatrix}
    1 & 1+\delta_U & 1+2\delta_U & \dots & 1+(r-2)\delta_U & 1+(r-1)\delta_U\\
    \delta_L & -\delta_U & -\delta_U & -\delta_U & \dots & -\delta_U\\
    2\delta_L & \delta_L-\delta_U & -2\delta_U & -2\delta_U &\dots & \vdots\\
    \ddots & \ddots & \ddots & \dots &\dots & \dots\\
    (r-2)\delta_L & \ddots & \ddots & \ddots &\dots & -(r-2)\delta_U\\
    (r-1)\delta_L & (r-2)\delta_L-\delta_U & \ddots & \ddots &\delta_L-(r-2)\delta_U & -(r-1)\delta_U
    \end{bmatrix}.
    $$
    \begin{enumerate}[label=(\alph*),wide, labelwidth=!, labelindent=0pt]
        \item For \cref{thmrwithP3}, let $\delta_U=0$, then
        $$
        \|(\tilde{A}_1(b))^{-1}\|_1 = r + \frac{r^3-r}{6}{\delta_L}.
        $$
        We have
        $$
        \lim_{\delta_L\rightarrow 0^+}\frac{\|\tilde{A}^{-1}\|_1}{\|(\tilde{A}_1(b))^{-1}\|_1}=r.
        $$
        \item For \cite[Theorem 9]{FampaLee2018ORL}, we then swap row $r+1$ ($[c^\top,d]$) with row $1$ in $\tilde{A}_1(b)$ to obtain $\tilde{A}_1(c,b)$. 
        By the Sherman-Morrison formula, we have
            {\small
        \begin{align*}
        (\tilde{A}_1(c,b))^{-1}
        &= \tilde{A}_1(b)^{-1} - \tilde{A}_1(b)^{-1}e_1(\mathbf{1}-e_1)^\top.
        \end{align*}
        }
        Thus $(\tilde{A}_1(c,b))^{-1}=$
    $$
    \begin{bmatrix}
    1 & \delta_U & 2\delta_U & \dots & (r-2)\delta_U & (r-1)\delta_U\\
    \delta_L & -(\delta_L+\delta_U) & -(\delta_L+\delta_U) & \dots & \dots & -(\delta_L+\delta_U)\\
    2\delta_L & -(\delta_L+\delta_U) & -2(\delta_L+\delta_U) & \dots &\dots & -2(\delta_L+\delta_U)\\
    \vdots & \vdots & \vdots & \ddots &\dots & \vdots\\
    (r-2)\delta_L & -(\delta_L+\delta_U) & -2(\delta_L+\delta_U) & \vdots &\ddots & -(r-2)(\delta_L+\delta_U)\\
    (r-1)\delta_L & -(\delta_L+\delta_U) & -2(\delta_L+\delta_U) & \vdots &-(r-2)(\delta_L+\delta_U) & -(r-1)(\delta_L+\delta_U).
    \end{bmatrix}
    $$
        Letting $\delta_U=0$, we have
        $$
        \|(\tilde{A}_1(c,b))^{-1}\|_1 = 1 + \frac{r^3-r}{3}{\delta_L},
        \mbox{ and then }
        \lim_{\delta_L\rightarrow 0^+}\frac{\|\tilde{A}^{-1}\|_1}{\|(\tilde{A}_1(c,b))^{-1}\|_1}=r^2.
        $$
        \item For \cref{thm:symapprox}, we choose $\delta_L=\delta_U$. Similarly, $\tilde{A}_1(b,b)$ is symmetric, and we compute
        $$
        \|(\tilde{A}_1(b,b))^{-1}\|_1 = 1 + \frac{r^3-r}{3}{(\delta_L+\delta_U)},
        \mbox{ and then }
        \lim_{\delta_L\rightarrow 0^+}\frac{\|\tilde{A}^{-1}\|_1}{\|(\tilde{A}_1(b,b))^{-1}\|_1}=r^2.
        $$
    \end{enumerate}
\end{example}
\subsection{Good case for local search based on the 1-norm of the inverse} \label{sec:r+1}
Now we consider the local search based on the 1-norm of the inverse. Here the local search is to find a local minimizer on the 1-norm of the inverse, which is defined similarly as the local-maximizer on the determinant. For example, for the general case, it is defined as
\smallskip
\begin{definition}
Let $A$ be an arbitrary $m\times n$, rank-$r$ matrix, and let $S$ be an ordered subset of $r$ elements from $\{1,\dots,m\}$ and $T$ an ordered subset of $r$ elements from $\{1,\dots,n\}$, and fixed $\epsilon\ge0$, if $\|(A[S,T])^{-1}\|_1$ cannot be decreased by either either swapping an element of $S$ with one from its complement or swapping an element of $T$ with one from its complement, then we say that $A[S,T]$ is a local minimizer for the 1-norm of the inverse on the set of $r\times r$ nonsingular submatrices of $A$.
\end{definition}
We prove an optimal approximation ratio $\frac{2r}{r+1}$ for $r$ by $r+1$ rank-$r$ matrices.

\begin{theorem}\label{thm:bound2}
    For a full row rank matrix $A\in\mathbb{R}^{r\times (r+1)}$, where $r:=\mathrm{rank}(A)$. If $\tilde{A}$ is chosen to minimize the 1-norm of $\tilde{A}^{-1}$ among all nonsingular $r\times r$ principal submatrices, then the $(r+1)\times r$ matrix $H$ constructed over $\tilde{A}$, is an ah-symmetric reflexive generalized inverse of $A$, satisfying $\|H\|_1\le \frac{2r}{r+1}\|H_{opt}\|_1$, where $H_{opt}$ is an optimal solution to $\min\{\|H\|_1:~\ref{property1}+\ref{property2}+\ref{property3}\}=\min\{\|H\|_1:~\ref{property1}\}=\min\{\|H\|_1:~AH=I_r\}$.
\end{theorem}
\begin{proof}
    Without loss of generality, assume that $A=[\tilde{A}~b]$, $b=Ax$, and $\{i:~x_i\ne 0\}=[s]$. Let $M:=\mathrm{sign}(\tilde{A}^{-1})$, and $W:=\begin{bmatrix}
\tilde{W} & 0
\end{bmatrix}
:=\begin{bmatrix}
\tilde{A}^{-\top}M\tilde{A}^{-\top}& 0
\end{bmatrix}.
$
We have
$$
A^\top W A^\top=\begin{bmatrix}
M\\
b^\top\tilde{A}^{-\top}M
\end{bmatrix}=A^\top \tilde{A}^{-\top}M.
$$
If $\|x\|_1=\sum_{i=1}^r|x_i|\le 1$, then $W$ is dual feasible, thus $H$ is also an optimal solution. We may assume that $\|x\|_1>1$.

Let $M_i(b):=\mathrm{sign}((\tilde{A}_i(b))^{-1})$ for $i\in[s]$, where $\tilde{A}_i(b)$ is $\tilde{A}$ with column $i$ replaced by $b$, and let
$$W_i(b):=\tilde{W}_i(b) D^\top:=(\tilde{A}_i(b))^{-\top} M_i(b)(\tilde{A}_i(b))^{-\top}D^\top,$$
where $D\in\mathbb{R}^{n\times m}$ with $AD = \tilde{A}_i(b)$.
The dual objective value for $W_i(b)$ is
$\langle A,W_i(b)\rangle=\mathrm{trace}(A^\top W_i(b))=\mathrm{trace}((\tilde{A}_i(b))^\top \tilde{W}_i(b))=\langle M_i(b),(\tilde{A}_i(b))^{-1}\rangle=\norm{(\tilde{A}_i(b))^{-1}}_1,$
i.e., $\langle A,W_i(b)\rangle = \norm{(\tilde{A}_i(b))^{-1}}_1\ge \norm{\tilde{A}^{-1}}_1$.
Also, we have
$
A^\top W_i(b) A^\top=A^\top \tilde{A}_i(b)^{-\top} M_i(b).
$
Now consider the dual solution $W_0 = \lambda W + \sum_{i=1}^s\lambda_{i} W_i(b)$ with $\lambda,\lambda_i\ge0$ and $\lambda+\sum_{i=1}^s\lambda_i=1$, which is a convex combination of $W$ and $W_i(b)$. We claim that $\min_{\lambda,\lambda_k}\|A^\top W_0 A^\top\|_{\max}\le\frac{2s}{s+1}\le\frac{2r}{r+1}$, which implies that there exists $W_0$ such that $\frac{r+1}{2r}W_0$ is dual feasible. Clearly, $\langle A,W_0\rangle\ge (\lambda+\sum_{i=1}^s\lambda_{i})\|\tilde{A}^{-1}\|_1=\|\tilde{A}^{-1}\|_1$. Thus $\|\tilde{A}^{-1}\|\le \frac{2r}{r+1}\|H_{opt}\|_1$. It remains to show that
$$\min_{\lambda,\lambda_k\ge0:~\lambda+\sum_{k=1}^s\lambda_k=1}\|A^\top W_0 A^\top\|_{\max}\le\frac{2s}{s+1}.$$
For simplicity, let $N:=\tilde{A}^{-1}$. By the Sherman-Morrison formula, we have
$$(\tilde{A}_i(b))^{-1}=\left(I - \frac{1}{x_{i}}(x-e_i)e_i^\top \right)N.$$
Thus $(\tilde{A}_i(b))^{-1}_{i\ell} =\frac{1}{x_{i}}N_{i\ell}$, and $(\tilde{A}_i(b))^{-1}_{k\ell} =\frac{1}{x_{i}}[x_{i}N_{k\ell}-x_{k}N_{i\ell}]$, $k\ne i$.\\
$[M_i(b)]_{il}=\frac{1}{\mathrm{sign}(x_{i})}M_{il}$, and $[M_i(b)]_{k\ell}=\frac{1}{\mathrm{sign}(x_{i})}\mathrm{sign}(x_{i}N_{k\ell}-x_{k}N_{i\ell})$, $k\ne i$.

\noindent
Because $\|\tilde{A}^{-1}\|_1\le\|(\tilde{A}_i(b))^{-1}\|$, we have
\begin{align*}
 &\|N\|_1\le \frac{1}{|x_i|}\|N_{i\cdot}\|_1+\sum_{k\ne i}\sum_\ell\frac{1}{|x_i|}|x_iN_{k\ell}-x_kN_{i\ell}|\\
 \Rightarrow~&|x_i|\|N\|_1\le \|N_{i\cdot}\|_1+\sum_{k\ne i}\sum_\ell|x_iN_{k\ell}-x_kN_{i\ell}|\\
 &\phantom{|x_i|\|N\|_1}\le \|N_{i\cdot}\|_1+|x_i|\sum_{k\ne i}\|N_{k\cdot}\|_1+\|N_{i\cdot}\|_1\sum_{k\ne i}|x_k|\\
 \Rightarrow~& 0\le (1+\sum_{k\ne i}|x_k|-|x_i|)\|N_{i\cdot}\|_1\Rightarrow 0\le 1+\sum_{k\ne i}|x_k|-|x_i|.
\end{align*}
Now, let $\bar{M} := \lambda M +\sum_{i=1}^s\lambda_{i}(I-\frac{1}{x_{i}}(x-e_i)e_i^\top)^\top M_i(b)$. Then
$$
A^\top W_0 A^\top=A^\top \tilde{A}^{-\top}\bar{M}=\begin{bmatrix}I\\ x^\top\end{bmatrix}\bar{M}=\begin{bmatrix}\bar{M}\\ x^\top\bar{M}\end{bmatrix}.
$$
For $k=[r]\setminus[s]$, we have $[\bar{M}]_{k\ell} = \lambda M_{k\ell} +\sum_{i=1}^s\lambda_{i}[M_i(b)]_{k\ell}$, thus $\|\bar{M}_{k\ell}\|\le 1$. \\
\noindent For $k\in[s]$, we have
\begin{align*}
&[\bar{M}]_{k\ell} = \lambda M_{k\ell} +\sum_{i\ne k:~i\in[s]}\lambda_{i}[M_i(b)]_{k\ell}+\lambda_{k}\left[\frac{[M_k(b)]_{k\ell}}{x_{k}}-\sum_{i\ne k}\frac{x_{i}}{x_{k}}[M_k(b)]_{i\ell}
\right]\\
&\quad= \lambda M_{k\ell} +\sum_{i\ne k:~i\in[s]}\left[\lambda_{i}\frac{1}{\mathrm{sign}(x_{i})}+\lambda_{k}\frac{x_{i}}{|x_{k}|}\right]\mathrm{sign}(x_{i}N_{k\ell}-x_{k}N_{i\ell})+\lambda_{k}\frac{M_{k\ell}}{|x_{k}|}\\
&\quad= \left(\lambda +\frac{\lambda_{k}}{|x_{k}|}\right)M_{k\ell} +\sum_{i\ne k:~i\in[s]}\left[\frac{\lambda_{i}}{|x_{i}|}+\frac{\lambda_{k}}{|x_{k}|}\right]x_{i}\mathrm{sign}(x_{i}N_{k\ell}-x_{k}N_{i\ell}).
\end{align*}
Therefore,
\begin{align*}
\|\bar{M}\|_{\max} &\le \max\left\{1,\max_{k\in[s]}\left\{\lambda +\frac{\lambda_{k}}{|x_{k}|} +\sum_{i\ne k:~i\in[s]}\left[\frac{\lambda_{i}}{|x_{i}|}+\frac{\lambda_{k}}{|x_{k}|}\right]|x_{i}|\right\}\right\}\\
& = \max\left\{1,\max_{k\in[s]}\left\{\lambda +\frac{1+\sum_{i\ne k:~i\in[s]}|x_{i}|}{|x_{k}|}\lambda_{k} +\sum_{i\ne k:~i\in[s]}\lambda_{i}\right\}\right\}.
\end{align*}
Also
\begin{align*}
&\quad(x^\top\bar{M})_{\ell} = \sum_{k=1}^r x_{k}[\bar{M}]_{k\ell}=\sum_{k=1}^s\left(\lambda +\frac{\lambda_{k}}{|x_{k}|}\right)x_{k}M_{k\ell}.
\end{align*}
Therefore
$$
\|x^\top\bar{M}\|_{\max}\le\sum\limits_{k=1}^s|x_{k}|\lambda + \sum_{k=1}^s\lambda_{k}.
$$
Next, we derive an upper bound for
$$
t:=\min_{\lambda,\lambda_{k}\ge0: \atop \lambda+\sum\limits_{k=1}^s\lambda_{k}=1} \max\left\{\sum\limits_{k=1}^s|x_{k}|\lambda + \sum_{k=1}^s\lambda_{k}, \lambda +\frac{1+\sum_{i\ne k:~i\in[s]}|x_{i}|}{|x_{k}|}\lambda_{k} +\sum_{i\ne k:~i\in[s]}\lambda_{i}\right\}
$$
$$
=\min_{\lambda,\lambda_{k}\ge0: \atop \lambda+\sum\limits_{k=1}^s\lambda_{k}=1} \max\left\{(\sum\limits_{k=1}^s|x_{k}|-1)\lambda +1, \frac{1+\sum_{i\ne k:~i\in[s]}|x_{i}|-|x_{k}|}{|x_{k}|}\lambda_{k} +1\right\}.
$$
Let $y_0 :=(\sum\limits_{k=1}^s|x_{k}|-1)>0$, and  $y_k:=\frac{1+\sum_{i\ne k:~i\in[s]}|x_{i}|-|x_{k}|}{|x_{k}|}\ge0$, $k\in[s]$.\\
\noindent
If $y_k=0$ for some $k\in[s]$, then $\lambda=0$, $\lambda_k=1$, $\lambda_i=0$ for $i\ne k$, is a feasible solution, and thus $t\le1$. If $y_k>0$ for $k\in[s]$, let $\frac{1}{y} := \frac{1}{y_0}+\sum_{k=1}^s\frac{1}{y_k}$. Then $\lambda:=\frac{y}{y_0}$, $\lambda_{y}:=\frac{y}{y_k}$ is a feasible solution, thus $t\le y+1$. Next, we seek an upper bound on $y+1$, which is equivalent to $\min_{x} \frac{1}{y}$. Letting $S :=\sum_{k=1}^s|x_{k}|$,
we have
\begin{align*}
   \frac{1}{y} &= \frac{1}{y_0}+\sum_{k=1}^s\frac{1}{y_k}=\frac{1}{\sum\limits_{k=1}^s|x_{k}|-1} + \sum_{k=1}^s\frac{|x_{k}|^2}{|x_{k}|(1+\sum_{i\ne k}|x_{i}|-|x_{k}|)}\\
   &\ge \frac{(1+S)^2}{S-1 + S(1+S)-2\sum_{k=1}^s|x_{k}|^2} \ge \frac{(1+S)^2}{S-1 + S(1+S)-2\frac{S^2}{s}}\\
   & = \frac{(1+S)^2}{(1-\frac{2}{s})S^2+2S-1} \ge \frac{s+1}{s-1}
\end{align*}
Therefore, $\min_{\lambda,\lambda_k\ge0}\|A^\top W_0 A^\top\|_{\max}\le 1+\frac{s-1}{s+1}=\frac{2s}{s+1}$.
\end{proof}
\begin{remark}
Note that when $A:=[\hat{A} ~~\hat{A}\mathbf{1}]$, where $\hat{A}=(J+rI)^{-1}$, the bound $\frac{2r}{r+1}$ is reached.
\end{remark}
\subsection{Bad case for local search based on the 1-norm of the inverse} \label{sec:1norm_ex}
For simplicity, we only consider full row rank matrix $A$ in this subsection, but the result can be extended to the symmetric case.
\begin{theorem}\label{thm:no_approx}
    There are no constant approximation ratio for local search  based on the 1-norm of the inverse for full row rank matrix $A\in\mathbb{R}^{r\times n}$, where $r\ge 2$ and $n\ge 2r$.
\end{theorem}
\begin{proof}
Let $A := \begin{bmatrix}1 & 0 & k & k\\ 0 & 1 & k & -k\end{bmatrix}$. So we have $r=2$.
Consider the block $A(:,[1,2])$, it has $\|A(:,[1,2])^{-1}\|_1=2$. Note that
    \begin{align*}
    &A(:,[1,3])^{-1} = \begin{bmatrix}1 & -1 \\ 0 & \frac{1}{k}\end{bmatrix},
    ~A(:,[1,4])^{-1} = \begin{bmatrix}1 & 1 \\ 0 & -\frac{1}{k}\end{bmatrix},\\
    &A(:,[2,3])^{-1} = \begin{bmatrix}-1 & 1 \\ \frac{1}{k} & 0\end{bmatrix},
    ~A(:,[2,4])^{-1} = \begin{bmatrix}1 & 1 \\ \frac{1}{k} & 0\end{bmatrix},
    \end{align*}
and we have $\|A(:,[1,3])^{-1}\|_1=\|A(:,[1,4])^{-1}\|_1=\|A(:,[2,3])^{-1}\|_1=\|A(:,[2,4])^{-1}\|_1=2+\frac{1}{k}>\|A(:,[1,2])^{-1}\|_1$, which implies that $A(:,[1,2])$ is a local-minimizer of $\|\tilde{A}^{-1}\|_1$. However, $\|(A(:,[3,4]))^{-1}\|_1 = \frac{2}{k}$, thus the approximation ratio is at least $k$. Because $k$ is a parameter which can be sent to infinity, thus there is no constant approximation ratio.

Now, let $\tilde{A}_n(a,b)$ be the $n$ by $n$ matrix with all entries equal to $b$ except the diagonal entries are equal to $a$, i.e. $\tilde{A}_n(a,b):=\mathrm{diag}((a-b)\mathbf{1}) + bJ$. Note that
$$\tilde{A}_n(a,b)\mathbf{1} = [a+(n-1)b]\mathbf{1}, \quad \det(\tilde{A}_n(a,b))=[a+(n-1)b](a-b)^{n-1},$$
and
$$
(\tilde{A}_n(a,b))^{-1} =\tilde{A}_n\left(\frac{-a-(n-2)b}{(b-a)(a+(n-1)b)},\frac{b}{(b-a)(a+(n-1)b)}\right).
$$
For $r\ge 3$. Let $\tilde{A}:=\tilde{A}_r(-1,1)$, and $A = [I_r ~k\tilde{A}]$ ($k>0$). Note that $\rank(\tilde{A})=r$ because $\det(\tilde{A})=(-2)^{r-1}(r-2)\ne 0$. Consider the block $I_r$; it has $\|I_r^{-1}\|_1=r$. If we replace any column of $I_r$ by a column of $k\tilde{A}$, then with some rearrangement of the columns and rows, we obtain
$$
\tilde{B}_1:=\begin{bmatrix}
I_{r-1} & k\mathbf{1} \\
0 & -k
\end{bmatrix}
\quad \text{or} \quad
\tilde{B}_2:=\begin{bmatrix}
I_{r-2} & 0 & k\mathbf{1} \\
0 & 1 & -k\\
0 & 0 & k
\end{bmatrix}.
$$
We have
$$
\tilde{B}_1^{-1} = \begin{bmatrix}
I_{r-1} & \mathbf{1} \\
0 & -\frac{1}{k}
\end{bmatrix}
, \quad \text{and} \quad
\tilde{B}_2^{-1}:=\begin{bmatrix}
I_{r-2} & 0 & -\mathbf{1} \\
0 & 1 & 1\\
0 & 0 & \frac{1}{k}
\end{bmatrix},
$$
thus $\|\tilde{B}_1^{-1}\|_1 = \|\tilde{B}_2^{-1}\|_1 =2(r-1) + \frac{1}{k}>\|I_r^{-1}\|_1$, which implies that $I_r$ is a local-minimizer of the 1-norm of the inverse. However, $\|(k\tilde{A})^{-1}\|_1 = \frac{r}{k}$, thus the approximation ratio is at least $k$. And $k$ is a parameter which can be sent to infinity, thus there is no constant approximation ratio.
\end{proof}

\subsection{Examples related to \texorpdfstring{\cref{prop:basicguarantee} part \ref{prop:basic123}}{Proposition 3.1 part (2)}} \label{sec:sb123_ex}
For $m>r$, $n=r^2$, $r\ge2$, we construct a family of examples that the unique optimal extreme solution of the LP for $\min\{\|H\|_1: P1+P2+P3\}$ have $mr+(r^2-r)(m-r) = r^2 + r^2(m-r)$ nonzeros. We consider the following optimization problem
\begin{equation*}\label{eqn:P123}\tag{\text{$\text{P}_{123}$}}
    \begin{array}{ll}
    \mbox{minimize }& \|H\|_1\\
    \mbox{subject to} & AHA = A\\
    &H(I_m - AA^+) = 0
    \end{array}
\end{equation*}
From $AHA=A$ and $H(I_m-AA^+)$, we could infer $AH=AH(I-AA^+)+AHAA^+ = (AHA)A^+=AA^+$, thus \eqref{eqn:P123} is equivalent to $\min\{\|H\|_1:~P1+P2+P3\}$. And we could derive the dual of \eqref{eqn:P123}:
\begin{equation*}\label{eqn:D123}\tag{\text{$\text{D}_{123}$}}
    \begin{array}{ll}
    \mbox{maximize }& \langle A,W\rangle\\
    \mbox{subject to} & -J\le A^\top WA^\top + V(I_m-AA^+)\le J.
\end{array}
\end{equation*}
Let $X=J_{r\times (m-r)}$, $Y\in\mathbb{R}^{r\times (r^2-r)}$, and the columns of $Y$ consist of all possible vectors $y\in\mathbb{R}^r$ with 2 nonzeros $\frac{m+r}{2m}$ and $\frac{m-r+1}{2m}$. Let $H_0=I_r$, $H_1\in\mathbb{R}^{(r^2-r)\times r}$, and $(H_1)_{ij} = 1$ if $Y_{ji} = \frac{m+r}{2m}$ otherwise $(H_1)_{ij}=0$. Then we have
\begin{align*}
    &Y^\top (\mathrm{sign}(H_0) + \mathrm{sign}(H_0X)X^\top)
    ~=~Y^\top (I_r + J_{r\times(m-r)}X^\top) ~=~ Y^\top (I_r + (m-r)J_r)\\
    &=~\mathrm{sign}(H_1) + D_1 + (m-r)J_{(r^2-r)\times r}
    ~=~\mathrm{sign}(H_1) + D_1 + \mathrm{sign}(H_1X)X^\top,
\end{align*}
where $D_1\in\mathbb{R}^{(r^2-r)\times r}$, $\|D_1\|_{\max}<1$, with
$$
(D_1)_{ij} = \left\{
    \begin{array}{ll}
        0, & Y_{ji}=\frac{m+r}{2m};\\
        \frac{2m-2r+1}{2m}, & Y_{ji}=\frac{m-r+1}{2m};\\
        \frac{m-r}{2m}, & Y_{ji}=0.
    \end{array}
\right.
$$
Let $A_0 = ((H_0+YH_1)(I_r+XX^\top))^{-1}$ and $A = \begin{bmatrix}I_r\\X^\top \end{bmatrix}A_0[I_r ~Y]\in\mathbb{R}^{m\times n}$,
$\rank(A)$ $=\rank(A_0)=r$.  Let $H=\begin{bmatrix}H_0\\ H_1\end{bmatrix}[I_r~X] = \begin{bmatrix}H_0 & J_{r\times (m-r)}\\ H_1 & J_{(r^2-r)\times (m-r)}\end{bmatrix}$,
$$
W = \begin{bmatrix}A_0^{-\top}W_0A_0^{-\top} & 0\\ 0 & 0\end{bmatrix}, ~W_0 = (\mathrm{sign}(H_0) + \mathrm{sign}(H_0X)X^\top)(I_r+XX^\top)^{-1},
$$
and
$$
V =\begin{bmatrix}\mathrm{sign}(H_0) &\mathrm{sign}(H_0X)\\ \mathrm{sign}(H_1)+D_1 & \mathrm{sign}(H_1X)\end{bmatrix} :=  \mathrm{sign}(H)+D~.
$$

These matrices $H,W,V$ satisfy
\begin{align*}
    &AA^+=\begin{bmatrix}I_r\\X^\top \end{bmatrix}(I_r+XX^\top)^{-1}[I_r ~X],\\
    &AHA = \begin{bmatrix}I_r\\X^\top \end{bmatrix}A_0[I_r ~Y]\begin{bmatrix}H_0 \\ H_1\end{bmatrix}[I_r ~X]\begin{bmatrix}I_r\\X^\top \end{bmatrix}A_0[I_r ~Y]\\
        &\quad=\begin{bmatrix}I_r\\X^\top \end{bmatrix}A_0(H_0+YH_1)(I_r+XX^\top)A_0[I_r ~Y]
        =\begin{bmatrix}I_r\\X^\top \end{bmatrix}A_0[I_r ~Y]=A,\\
    &H(I_m-AA^+) = H\begin{bmatrix}X \\ -I_{m-r}\end{bmatrix}(I_{m-r}+X^\top X)^{-1}[X^\top ~-I_{m-r}]\\
        &\quad=\begin{bmatrix}H_0 \\ H_1\end{bmatrix}[I_r ~X]\begin{bmatrix}X \\ -I_{m-r}\end{bmatrix}(I_{m-r}+X^\top X)^{-1}[X^\top ~-I_{m-r}]
        =0,\\
    &A^\top WA^\top + V(I_m-AA^+)=\begin{bmatrix}I_r\\Y^\top \end{bmatrix}W_0[I_r ~X] + V(I_m-AA^+)\\
        &\quad= \begin{bmatrix}I_r\\Y^\top \end{bmatrix}(\mathrm{sign}(H_0)+\mathrm{sign}(H_0X)X^\top)(I_r+XX^\top)^{-1}[I_r ~X]
           + V(I_m-AA^+)\\
        &\quad= \begin{bmatrix}\mathrm{sign}(H_0)+\mathrm{sign}(H_0X)X^\top\\\mathrm{sign}(H_1)+D_1+\mathrm{sign}(H_1X)X^\top \end{bmatrix}(I_r+XX^\top)^{-1}[I_r ~X]
         + V(I_m-AA^+)\\
        &\quad= (\mathrm{sign}(H)+D)AA^+ + V(I_m-AA^+)
        =\mathrm{sign}(H)+D,\\
    &\langle A,W\rangle = \langle AHA,W\rangle=\langle H, A^\top WA^\top\rangle
    =\langle H, A^\top WA^\top\rangle+ \langle H(I_m-AA^+),V\rangle\\
    &\quad=\langle H, A^\top WA^\top + V(I_m-AA^+)\rangle
    =\langle H, \mathrm{sign}(H)+D\rangle
    =\|H\|_1~.
\end{align*}
Therefore by  weak duality, $H$ and $W,V$ are  optimal primal and dual solutions, and $H$ has exactly $r^2+r^2(m-r)$ nonzeros. Also, because $\mathrm{vec}(A^\top WA^\top+V(I_m-AA^+))$ has exactly $r^2+r^2(m-r)$ entries with value $\pm1$ corresponding to the positions where $\mathrm{vec}(H)$ is nonzero, by complementary slackness we have that for any primal optimal solution $H^*$, $\mathrm{vec}(H^*)$ is nonzero only in positions where $\mathrm{vec}(H)$ is nonzero. Then we can easily solve the system of equations $AHA=A, H(I_m-AA^+)=0$ to obtain the unique solution $\mathrm{vec}(H)$. Therefore the primal problem has a unique optimal extreme solution $H$ with $r^2+r^2(m-r)$ nonzeros.
\end{document}